\documentclass[a4paper,centertags, reqno]{amsart}
\usepackage{amssymb}
\usepackage[mathcal]{euscript}

\numberwithin{subsection}{section}
\numberwithin{equation}{subsection}

\newtheorem{thm}{Theorem}[subsection]
\newtheorem{lemma}[thm]{Lemma}
\newtheorem{assumption}[thm]{Assumption}

\newtheorem{cor}[thm]{Corollary}
\newtheorem{example}[thm]{Example}
\newtheorem{prop}[thm]{Proposition}
\newtheorem{convention}[thm]{Convention}
\newtheorem{problem}[thm]{Problem}

\theoremstyle{definition}
\newtheorem{definition}[thm]{Definition}

\theoremstyle{remark}
\newtheorem*{rem*}{Remark}
\newtheorem{rem}[thm]{Remark}

\newcommand{\mr}{{\mathbb R}}

\newcommand{\mn}{{\mathbb N}}
\newcommand{\mz}{{\mathbb Z}}
\newcommand{\mc}{{\mathbb C}}
\newcommand{\md}{{\mathbb D}}

\newcommand{\mt}{\mathbb{T}}

\newcommand{\eps}{\varepsilon}
\newcommand{\vphi}{\varphi} \renewcommand{\Im}{\operatorname{Im}}

\renewcommand{\Re}{\operatorname{Re}}
\renewcommand{\Im}{\operatorname{Im}}

\renewcommand{\ker}{\operatorname{Ker}}

\newcommand{\ran}{\operatorname{Ran}}

\newcommand{\sign}{\operatorname{sign}}

\newcommand{\hil}{\mathcal{H}}

\newcommand{\supp}{\operatorname{supp}}
\newcommand{\dist}{\operatorname{dist}}
\newcommand{\rank}{\operatorname{Rank}}
\newcommand{\num}{\operatorname{Num}}

\newcommand{\dom}{\operatorname{Dom}}
\newcommand{\detp}{\operatorname{det}_{\lceil p \rceil}}
\newcommand{\bdd}{\mathcal{B}}
\newcommand{\cld}{\mathcal{C}}

\newcommand{\dotcup}{\mathbin{\dot{\cup}}}

\begin{document}

\title[Eigenvalues of non-selfadjoint operators]{Eigenvalues of non-selfadjoint operators: A comparison of two approaches}

\author[M. Demuth]{Michael Demuth}
\address{Institute for Mathematics\\
Clausthal University of Technology\\
Clausthal\\
Germany.}
\email{demuth@math.tu-clausthal.de}

\author[M. Hansmann]{Marcel Hansmann}
\address{Faculty of Mathematics\\
Chemnitz University of Technology\\
Chemnitz\\
Germany.}
\email{marcel.hansmann@mathematik.tu-chemnitz.de}

\author[G. Katriel]{Guy Katriel}
\address{Department of Mathematics\\
Braude College\\
Karmiel\\
 Israel}
\email{katriel@braude.ac.il}

\begin{abstract}
The central problem we consider is the distribution of eigenvalues of closed linear operators which are not selfadjoint, with a focus on those operators
which are obtained as perturbations of selfadjoint linear operators. Two methods are explained and elaborated. One approach uses complex analysis to
study a holomorphic function whose zeros can be identified with the eigenvalues of the linear operator. The second method is an operator
theoretic approach involving the numerical range. General results obtained by the two
methods are derived and compared. Applications to non-selfadjoint  Jacobi and Schr\"odinger operators are considered. Some possible directions for future
research are discussed.
\end{abstract}

\maketitle

\tableofcontents

\section{Introduction}

\numberwithin{equation}{section}

The importance of eigenvalues and eigenvectors is clear to every student of mathematics, science or engineering. As a simple example, consider a linear dynamical system which is described by an equation of the form
\begin{equation}\label{linear}u_t=Lu,\end{equation}
where $u(t)$ is an element in a linear space $X$ and $L$ a linear operator in $X$. If we can find an eigenpair $v\in X$, $\lambda \in \mc$ with
$Lv=\lambda v$, then we have solved (\ref{linear}) with the initial condition $u(0)=v$: $u(t)=e^{\lambda t}v$. If we can find a whole basis
of eigenvectors, we have solved (\ref{linear}) for any initial condition $u(0)=u_0$ by decomposing $u_0$ with respect to this basis. So the knowledge of the eigenvalues of $L$
(or more generally, the analysis of its spectrum) is essential for the understanding of the corresponding system.

The \textit{spectral analysis} of linear operators has a quite long history, as everybody interested in the field is probably aware of. Still, we think that it can be worthwhile to begin this introduction with a short historical survey, which  will also help to put the present article in its proper perspective. The origins of spectral analysis  can be traced back at least as far as the  work of D'Alembert and Euler (1740-50's) on vibrating strings, where
eigenvalues correspond to frequencies of vibration, and eigenvectors correspond to modes of vibration. When the vibrating string's density and tension is not uniform,
the eigenvalue problem involved becomes much more challenging, and an early landmark of spectral theory is Sturm and Liouville's (1836-1837) analysis of general one-dimensional
problems on bounded intervals, showing the existence of an infinite sequence of eigenvalues. This naturally gave rise to questions about corresponding
results for differential operators on higher-dimensional domains, with the typical problem being the eigenvalues of the Laplacian on a bounded domain with Dirichlet boundary conditions. The existence of the first eigenvalue for this problem was obtained by Schwartz (1885), and of the the second eigenvalue by Picard (1893), and
it was Poincar\'e (1894) who obtained existence of all eigenvalues and their basic properties.  Inspired by Poincar\'es work, Fredholm (1903) undertook
the study of the spectral theory of integral operators. Hilbert (1904-1910), generalizing the work of Fredholm, introduced the ideas of quadratic forms
 on infinite dimensional linear spaces and of completely continuous forms (compact operators in current terminology).
He also realized that spectral analysis cannot be performed in terms of eigenvalues alone, developing the notion of continuous spectrum, which was prefigured in
Wirtinger's (1897) work on Hill's equation. Weyl's (1908) work on integral equations on unbounded intervals further stresses the importance of the
continuous spectrum.  The advent of quantum mechanics, formulated axiomatically by von Neumann (1927), who was the first to introduce the notion of an abstract Hilbert space,
brought selfadjoint operators into the forefront of interest.
 Kato's \cite{MR0041010} rigorous proof of the
selfadjointness of physically relevant Schr\"odinger operators was a starting point for the mathematical study of particular operators. In the
context of quantum mechanics, eigenvalues have special significance, as they correspond to discrete energy levels, and thus form the basis for the quantization
phenomenon, which in the pre-Schr\"odinger quantum theory had to be postulated a-priori. In recent years, non-selfadjoint operators are also becoming increasingly
important in the study of quantum mechanical systems, as they arise naturally in, e.g., the optical model of nuclear scattering or  the study of the behavior of unstable lasers (see \cite{b_Davies07} and references therein).
 
As this brief sketch\footnote{The interested reader can find much more information (and detailed references) in Mawhin's account \cite{Mawhin} on the origins of spectral analysis.} of some highlights of the (early) history of spectral theory shows, eigenvalues, eigenvectors, and the spectrum provide an endless source of fascination for both mathematicians and physicists.
At the most general level one may ask, given a class $\cld$ of linear operators (which in our case will always operate in a Hilbert space), what can be said about the spectrum of operators $L\in \cld$? Of course, the more restricted is
the class of operators considered the more we can say, and the techniques available for studying different classes $\cld$ can vary enormously.
For example, an important part of the work of Hilbert is a theory of selfadjoint compact operators, which in
particular characterizes their spectrum as an infinite sequence of real eigenvalues. Motivated by various applications, this class of operators can be restricted or broadened to yield other classes worth studying. For example, the study of eigenvalues of the Dirichlet problem in a bounded domain is a restriction of the class of compact selfadjoint eigenvalue problems, which yields a rich theory relating the eigenvalues to the geometrical properties of the domain in question. As far as broadening the class of operators goes, one can consider selfadjoint operators which are not compact, leading to a vast domain of study which is of great importance to a variety of areas of application, perhaps the most prominent being quantum mechanics.  One can also
consider compact operators which are not selfadjoint (and which might act in general Banach spaces), leading to a field of research in which natural sub-classes of the class of compact operators
are defined and their sets of eigenvalues are studied (see e.g. the classical works of Gohberg and Krein \cite{b_Gohberg69} or Pietsch \cite{MR917067}).

One may also lift both the assumption of selfadjointness {\it{and}} that of compactness. However, {\it{some}} restriction on the class of
operators considered must be made in order to be able to say anything nontrivial about the spectrum. The classes of operators that we will
be considering here are those that arise by perturbing bounded or unbounded (in most cases selfadjoint) operators  with no isolated eigenvalues by operators which are (relatively) compact, for example operators of the form $A=A_0+M$, where $A_0$ is a bounded operator with spectrum $\sigma(A_0)=[a,b]$ and $M$ is a compact operator in a
certain Schatten class. More precisely, we will be interested in the isolated eigenvalues of such operators $A$ and in their rate of accumulation to the essential spectrum $[a,b]$. We will study this rate by analyzing eigenvalue moments of the form 
\begin{equation}
  \label{eq:AA1}
 \sum_{\lambda\in \sigma_d(A)} (\dist(\lambda,[a,b]))^p, \qquad p > 0,
\end{equation}
where $\sigma_d(A)$ is the set of discrete eigenvalues, and by bounding these moments in terms of the Schatten norm of the perturbation $M$.

It is well known that the summation of two `simple' operators can generate an operator whose spectrum is quite difficult to understand, even in case that both operators are selfadjoint. In our case, at least one of the operators will be non-selfadjoint, so the huge toolbox of the selfadjoint theory (containing, e.g., the spectral theorem, the decomposition of the spectrum into its various parts or the variational characterization of the eigenvalues) will not be available. This will make the problem even more demanding and also indicates that we cannot expect to obtain as much information on the spectrum as in the selfadjoint case. At this point we cannot resist quoting E.~B.~Davies, who in the preface of his book \cite{b_Davies07} on the spectral theory of non-selfadjoint operators described the differences between the selfadjoint and the non-selfadjoint theory:  "Studying non-selfadjoint operators is like being a vet rather than a doctor: one has to acquire a much wider range of knowledge, and to accept that one cannot expect to have as high a rate of success when confronted with particular cases".\\

\noindent In our previous work, which we review in this paper, we have developed and explored two quite different approaches to obtain results on the
distribution of eigenvalues of non-selfadjoint operators. One approach, which has also benefitted from (and relies heavily on) some related work of Borichev, Golinskii and Kupin \cite{Borichev08}, involves the construction of a holomorphic function whose zeros coincide with the eigenvalues of the operator of interest (the `perturbation determinant') and the
study of these zeros by employing results of complex analysis. The second is an operator-theoretic approach using the concept
of numerical range. One of our main aims in this paper is to present these two methods side by side, and to examine the advantages of each of them
in terms of the results they yield. We shall see that each of these methods has certain advantages over the other.

The plan of this paper is as follows. In Chapter \ref{cha:preliminaries} we recall fundamental concepts and results of functional analysis and operator theory that will be used. In Chapter \ref{cha:complex} we discuss results on zeros of complex functions that will later be used to obtain results on eigenvalues. In particular, we begin this chapter with a short explanation why results from complex analysis can be used to obtain estimates on eigenvalue moments of the form (\ref{eq:AA1}) in the first place. Next, in Chapter \ref{sec:complex}, we develop the complex-analysis approach to obtaining results on eigenvalues of perturbed operators, obtaining results of varying degrees of
generality for Schatten-class perturbations of selfadjoint bounded operators and for relatively-Schatten perturbations of non-negative operators.
A second, independent, approach to  obtaining eigenvalue estimates via operator-theoretic arguments is exposed in Chapter \ref{sec:operator}, and applied
to the same classes of operators. In Chapter \ref{sec:comparison} we carry out a detailed comparison of the results obtained by the two approaches
in the context of Schatten-perturbations of bounded selfadjoint operators. In Chapter \ref{sec:applications} we turn to applications of the
results obtained in Chapter \ref{sec:complex} and  \ref{sec:operator} to some concrete classes of operators, which allows us to further
compare the results obtained by the two approaches in these specific contexts. We obtain results on the eigenvalues of Jacobi operators and of
Schr\"odinger operators with complex potentials. These case-studies also give us the opportunity to compare the results obtained by our
methods to results which have been obtained by other researchers using different methods. These comparisons give rise to some conjectures and open
questions which we believe could stimulate further research. Some further directions of ongoing work related to the work discussed in this paper,
and issues that we believe are interesting to address, are discussed in Chapter \ref{sec:outlook}.

\section{Preliminaries}\label{cha:preliminaries}

\numberwithin{equation}{subsection}

In this chapter we will introduce and review some basic concepts of operator and spectral theory, restricting ourselves to those aspects of the theory which are relevant in the later parts of this work. We will also use this chapter to set  our notation and terminology.  As general references let us mention the  monographs of Davies \cite{b_Davies07},  Gohberg, Goldberg and Kaashoek \cite{b_Gohberg90}, Gohberg and Krein \cite{b_Gohberg00} and Kato \cite{b_Kato95}.

\subsection{The spectrum of linear operators}

Let  $\hil$ denote a  complex separable Hilbert space and let  $Z$ be  a linear operator in $\hil$. The domain, range and kernel of $Z$ are denoted by $\dom(Z)$, $\ran(Z)$ and $\ker(Z)$, respectively .
We say that $Z$ is an operator \textit{on} $\hil$ if $\dom(Z)=\hil$. The algebra of all bounded operators on $\hil$ is denoted by $\bdd(\hil)$. Similarly, $\cld(\hil)$ denotes the class of all closed operators in $\hil$.  

In the following we assume that $Z$ is a closed operator in $\hil$. The resolvent set of $Z$ is defined as 
\begin{equation}
  \rho(Z):=\{ \lambda \in \mc : \lambda-Z \text{ is invertible in } \bdd(\hil) \}\footnote{Note that here and elsewhere in the text, we use $\lambda-Z$ as a shorthand for $\lambda I - Z$ where $I$ denotes the identity operator on $\hil$.}
\end{equation}
and for $\lambda \in \rho(Z)$ we define
\begin{equation}
  \label{eq:1}
R_Z(\lambda):=(\lambda-Z)^{-1}.
\end{equation}
The complement of $\rho(Z)$ in $\mc$, denoted by $\sigma(Z)$, is called the spectrum of $Z$. Note that $\rho(Z)$ is an open and $\sigma(Z)$ is a closed subset of $\mc$.  We say that $\lambda \in \sigma(Z)$ is an eigenvalue of $Z$ if $\ker(\lambda-Z)$ is nontrivial.

The extended resolvent set of $Z$ is defined as
\begin{equation}\label{eq:2}
  \hat{\rho}(Z):= \left\{
    \begin{array}{cl}
      \rho(Z) \cup \{ \infty \}, & \text{ if } Z \in \bdd(\hil) \\[4pt]
      \rho(Z), & \text{ if } Z \notin \bdd(\hil).
    \end{array}\right.
\end{equation}
In particular, if $Z \in \bdd(\hil)$ we regard $\hat{\rho}(Z)$ as a subset of the extended complex plane $\hat \mc = \mc \cup \{ \infty \}$. Setting $R_Z(\infty):=0$ if $Z \in \bdd(\hil)$, the operator-valued function
$$R_Z : \lambda \mapsto R_Z(\lambda),$$
called the resolvent of $Z$, is analytic on $\hat{\rho}(Z)$. Moreover, for every $\lambda \in \hat{\rho}(Z)$ the resolvent satisfies the inequality
$ \|R_Z(\lambda)\| \geq \dist(\lambda, \sigma(Z))^{-1}$, where $\|.\|$ denotes the norm of $\bdd(\hil)$\footnote{We will use the same symbol to denote the norm on $\hil$.}  and we agree that $1 / \infty := 0$. Actually, if $Z$ is a normal operator (that is, an operator commuting with its adjoint) then the spectral theorem implies that
\begin{equation}\label{eq:3}
\|R_Z(\lambda)\|= \dist(\lambda, \sigma(Z))^{-1}, \quad \lambda \in \hat{\rho}(Z).
\end{equation}

If $\lambda \in \sigma(Z)$ is an isolated point of the spectrum, we define the Riesz projection of $Z$ with respect to $\lambda$ by
\begin{equation}\label{eq:4}
   P_{Z}(\lambda) := \frac{1}{2\pi i} \int_\gamma R_Z(\mu) d\mu,
\end{equation}
where the contour $\gamma$ is a counterclockwise oriented circle centered at $\lambda$, with sufficiently small radius (excluding the rest of $\sigma(Z)$). 
We recall that a subspace $M \subset \hil$ is called $Z$-invariant if $Z(M \cap \dom(Z)) \subset M$. In this case, $Z|_M$ denotes the restriction of $Z$ to  $M \cap \dom(Z)$ and the range of $Z|_M$ is a subspace of $M$.

\begin{prop}[see, e.g., \cite{b_Gohberg90}, p.326] 
Let $Z \in \cld(\hil)$ and let $\lambda \in \sigma(Z)$ be isolated. If $P=P_Z(\lambda)$ is defined as above, then the following holds:
  \begin{enumerate}
  \item[(i)] $P$ is a projection, i.e., $P^2=P$.
  \item[(ii)] $\ran(P)$ and $\ker(P)$ are  $Z$-invariant.
  \item[(iii)] $\ran(P) \subset \dom(Z)$ and $Z|_{\ran(P)}$ is bounded.
  \item[(iv)]  $\sigma(Z|_{\ran(P)})=\{\lambda\}$ and $\sigma(Z|_{\ker(P)})=\sigma(Z) \setminus \{\lambda\}$.
  \end{enumerate}
\end{prop}

We say that $\lambda_0 \in \sigma(Z)$ is a discrete eigenvalue if $\lambda_0$ is an isolated point of $\sigma(Z)$ and $P=P_Z(\lambda_0)$ is of finite rank (in the literature these eigenvalues are also  referred to as "eigenvalues of finite type"). Note that in this case $\lambda_0$ is indeed an eigenvalue of $Z$ since $\{\lambda_0\}=\sigma(Z|_{\ran(P)})$ and $\ran(P)$ is $Z$-invariant and finite-dimensional. The positive integer
\begin{equation}
  \label{eq:5}
 m_Z(\lambda_0):=\rank(P_Z(\lambda_0))
\end{equation}
is called the algebraic multiplicity of $\lambda_0$ with respect to $Z$. It has to be distinguished from the geometric multiplicity, which is defined as the dimension of the eigenspace  $\ker(\lambda_0-Z)$ (and so can be smaller than the algebraic multiplicity).
\begin{convention}
In this article only algebraic multiplicities will be considered and we will use the term "multiplicity" as a synonym for "algebraic multiplicity".
\end{convention}
\noindent The  discrete spectrum of $Z$ is now defined as
\begin{equation}
  \label{eq:6}
  \sigma_d(Z):=\{ \lambda \in \sigma(Z): \lambda \text{ is a discrete eigenvalue of } Z \}.\\
\end{equation}

We recall that a linear operator $Z_0 \in \cld(\hil)$  is a Fredholm operator if it has closed range and both its kernel and cokernel are finite-dimensional.  Equivalently, if $Z_0 \in \cld(\hil)$ is densely defined, then $Z_0$ is Fredholm if it has closed range and both $\ker(Z_0)$ and $\ker(Z_0^*)$ are finite-dimensional.
The essential spectrum of $Z$ is defined as
\begin{equation}\label{eq:7}
  \sigma_{ess}(Z):=\{ \lambda \in \mc: \lambda- Z \text{ is not a Fredholm operator } \}.
\footnote{For a discussion of various alternative (non-equivalent) definitions of the essential spectrum we refer to \cite{b_Edmunds87}. We note that all reasonable definitions coincide in the selfadjoint case.}
\end{equation}
Note that $\sigma_{ess}(Z) \subset \sigma(Z)$ and that $\sigma_{ess}(Z)$ is a closed set.

For later purposes we will need the following result about the spectrum of the resolvent of $Z$ .
\begin{prop}[\cite{b_Engel00}, p.243 and p.247, and \cite{b_Davies07}, p.331] \label{prop:1}
Suppose that  $Z \in \cld(\hil)$ with $\rho(Z) \neq \emptyset$. If $a \in \rho(Z)$, then
  \begin{equation*}
    \sigma(R_Z(a)) \setminus \{0\} = \{ (a-\lambda)^{-1}: \lambda \in \sigma(Z) \}.
  \end{equation*}
The same identity holds when, on both sides, $\sigma$ is replaced by $\sigma_{ess}$ and $\sigma_d$, respectively.
More precisely, $\lambda_0$ is an isolated point of $\sigma(Z)$ if and only if $(a-\lambda_0)^{-1}$ is an isolated point of  $\sigma(R_Z(a))$ and in this case
\begin{equation*}
  P_Z(\lambda_0)=P_{R_Z(a)}((a-\lambda_0)^{-1}).
\end{equation*}
In particular, the algebraic multiplicities of $\lambda_0 \in \sigma_d(Z)$ and $(a-\lambda_0)^{-1} \in \sigma_d(R_Z(a))$ coincide.
\end{prop}

\begin{rem}\label{rem:3}
We note that $0 \in \sigma(R_Z(a))$ if and only if $Z \notin \bdd(\hil)$. Moreover, if $Z \in \cld(\hil)$ is densely defined, then
\begin{equation*}
  0 \in \sigma(R_Z(a)) \quad \Leftrightarrow \quad 0 \in \sigma_{ess}(R_Z(a)).
\end{equation*}
\end{rem}
The following proposition shows that the essential and the discrete spectrum of a linear operator are disjoint.
\begin{prop}
  If $Z \in \cld(\hil)$ and  $\lambda$ is an isolated point of $\sigma(Z)$, then $\lambda \in \sigma_{ess}(Z)$ if and only if $\:\rank(P_Z(\lambda))= \infty$. In particular,
  \begin{equation*}
    \sigma_{ess}(Z)\cap \sigma_d(Z) = \emptyset.
  \end{equation*}
\end{prop}
\begin{proof}
For $Z \in \bdd(\hil)$ a proof can be found in \cite{b_Davies07}, p.122. The unbounded case can be reduced to the bounded case by means of Proposition \ref{prop:1}.
\end{proof}

While the spectrum of a selfadjoint operator $Z$ can always be decomposed as
\begin{equation}
  \label{eq:8}
\sigma(Z) = \sigma_{ess}(Z) \dotcup \sigma_d(Z),
\end{equation}
where the symbol $\dotcup$ denotes a disjoint union, the same need not be true in the non-selfadjoint case. For instance, considering the shift operator $(Zf) (n)= f(n+1)$ acting on $l^2(\mn)$, we have $\sigma_{ess}(Z)=\{ z \in \mc : |z| = 1\}$ and $\sigma(Z)=\{ z \in \mc : |z| \leq 1 \}$, while $\sigma_d(Z) = \emptyset$, see \cite{b_Kato95}, p.237-238.
The following result gives a suitable criterion for the discreteness of the spectrum in the complement of $\sigma_{ess}(Z)$.
\begin{prop}[\cite{b_Gohberg90}, p.373]\label{prop:2}
  Let $Z \in \cld(\hil)$ and let $\Omega \subset \mc \setminus \sigma_{ess}(Z)$ be open and connected. If $\Omega \cap \rho(Z) \neq \emptyset$, then $\sigma(Z) \cap \Omega \subset \sigma_d(Z)$.
\end{prop}
Hence, if $\Omega$ is a (maximal connected) component of  $\mc \setminus \sigma_{ess}(Z)$, then either
\begin{enumerate}
\item[(i)] $\Omega \subset \sigma(Z)$ (in particular, $\Omega \cap \sigma_d(Z) = \emptyset$), or
\item[(ii)] $\Omega \cap \rho(Z) \neq \emptyset$ and $\Omega \cap \sigma(Z)$ consists of an at most countable sequence of discrete eigenvalues  which can accumulate at $\sigma_{ess}(Z)$ only.
\end{enumerate}
A direct consequence of Proposition \ref{prop:2} is
\begin{cor}\label{cor:a1}
  Let $Z \in \cld(\hil)$ with $\sigma_{ess}(Z) \subset \mr$ and assume that there are points of $\rho(Z)$ in both the upper and lower half-planes.
 Then $\sigma(Z)= \sigma_{ess}(Z) \dotcup \sigma_d(Z)$.
\end{cor}

We conclude this section with some remarks on the numerical range of a linear operator and its relation to the spectrum, see \cite{MR1417493}, \cite{b_Kato95} for extensive accounts on this topic. The numerical range of $Z \in \cld(\hil)$ is defined as
\begin{equation}
  \num(Z) := \{ \langle Zf,f \rangle : f \in \dom(Z), \|f\|=1 \}.
\end{equation}
It was shown by Hausdorff and Toeplitz (see, e.g., \cite{b_Davies07} Theorem 9.3.1) that the numerical range is always a convex subset of $\mc$. Furthermore, if the complement of the closure of the numerical range is connected and contains at least one point of the resolvent set of $Z$, then  $\sigma(Z) \subset \overline{\num}(Z)$ and
\begin{equation}\label{sec1:num}
  \|R_Z(a)\| \leq 1/\dist(a, \overline{\num}(Z)), \qquad a \in \mc \setminus \overline{\num}(Z).
\end{equation}
Clearly, if $Z \in \bdd(\hil)$ then $\num(Z) \subset \{ \lambda : |\lambda| \leq \|Z\| \}$. Moreover, if $Z$ is normal then the closure of $\num(Z)$ coincides with the convex hull of $\sigma(Z)$, i.e. the smallest convex set containing $\sigma(Z)$.

\subsection{Schatten classes and determinants}\label{sec:schatten}

An operator $K \in \bdd(\hil)$ is called compact if it is the norm limit of finite rank operators.  The class of all compact operators forms a two-sided ideal in
$\bdd(\hil)$, which we denote by $\mathcal{S}_\infty(\hil)$. The non-zero elements of the spectrum of $K \in \mathcal{S}_\infty(\hil)$ are discrete eigenvalues. In particular, the only possible accumulation point of the spectrum is $0$, and $0$ itself may or may not belong to the spectrum. More precisely, if $\hil$ is infinite-dimensional, as will be the case in most of the applications below, then $\sigma_{ess}(K)=\{0\}$.

For every $K \in \mathcal{S}_\infty(\hil)$ we can find (not necessarily complete) orthonormal sets $\{\phi_n\}$ and $\{\psi_n\}$ in $\hil$, and a set of positive numbers $\{s_n(K)\}$ with $s_1(K) \geq s_2(K) \geq \ldots > 0$, such that
\begin{equation}
  Kf =\sum_{n} s_n(K) \langle f, \psi_n \rangle \phi_n, \qquad f \in \hil.
\end{equation}
Here the numbers $s_n(K)$ are called the singular values of $K$. They are precisely the eigenvalues of $|K|:=\sqrt{K^*K}$, in non-increasing order.

The Schatten class of order $p$ (with $p \in (0,\infty)$), denoted by $\mathcal{S}_p(\hil)$, consists of all compact operators on $\hil$ whose singular values are $p$-summable, i.e.
\begin{equation}
  \label{eq:11}
 K \in \mathcal{S}_p(\hil) \quad :\Leftrightarrow \quad  \{s_n(K)\} \in l^p(\mn).
\end{equation}

We remark that $\mathcal{S}_p(\hil)$ is a linear subspace of $\mathcal{S}_\infty(\hil)$ for every $p>0$ and for $p \geq 1$ we can make it into  a complete normed space by setting
\begin{equation}
  \label{eq:12}
 \| K \|_{\mathcal{S}_p}:= \|\{s_n(K)\}\|_{l^p}.
\end{equation}
Note that for $0<p<1$ this definition provides only a quasi-norm. For consistency we set $\|K\|_{S_\infty}:=\|K\|$.

For  $0 < p < q \leq \infty$ we have the (strict) inclusion $\mathcal{S}_p(\hil) \subset \mathcal{S}_q(\hil)$ and
\begin{equation}
  \label{eq:184}
  \|K\|_{\mathcal{S}_q} \leq \|K\|_{\mathcal{S}_p}.
\end{equation}

Similar to the class of compact operators, $\mathcal{S}_p(\hil)$ is a two-sided ideal in the algebra $\bdd(\hil)$ and for $K \in \mathcal{S}_p(\hil)$ and $B \in \bdd(\hil)$ we have
\begin{equation}
  \label{eq:14}
  \|KB\|_{\mathcal{S}_p} \leq \|K\|_{\mathcal{S}_p} \|B\| \quad \text{and} \quad   \|BK\|_{\mathcal{S}_p} \leq \|B \| \|K\|_{\mathcal{S}_p}.
\end{equation}
Moreover, if $K \in \mathcal{S}_p(\hil)$ then  $K^* \in \mathcal{S}_p(\hil)$ and $\|K^*\|_{\mathcal{S}_p}=\|K\|_{\mathcal{S}_p}$.

The following estimate is a Schatten class analog of H\"older's inequality (see \cite{b_Gohberg00}, p.88):  Let $K_1 \in \mathcal{S}_p(\hil)$ and $K_2 \in \mathcal{S}_q(\hil)$ where $0<p,q\leq \infty$. Then $K_1K_2 \in \mathcal{S}_{r}(\hil)$, where $r^{-1}=p^{-1}+q^{-1}$, and
$$ \|K_1K_2\|_{\mathcal{S}_r} \leq \| K_1 \|_{\mathcal{S}_p} \|K_2\|_{\mathcal{S}_q}.$$

While the singular values of a selfadjoint operator are just the absolute values of its eigenvalues, in general the eigenvalues and singular values need not be related. However, we have the following result  of Weyl.

\begin{prop}\label{prop:6}
  Let $K \in \mathcal{S}_p(\hil),$ where $0<p<\infty,$ and let $\lambda_1, \lambda_2, \ldots$ denote its sequence of nonzero eigenvalues (counted according to their multiplicity). Then
  \begin{equation}
    \label{eq:16}
    \sum_{n} |\lambda_n|^p \leq \sum_{n} s_n(K)^p.
  \end{equation}
\end{prop}

In the remaining part of this section we will introduce the notion of an infinite determinant. To this end, let $K \in \mathcal{S}_n(\hil)$, where $n\in \mn$, and let $\lambda_1, \lambda_2, \ldots$  denote its sequence of nonzero eigenvalues, counted according to their multiplicity and enumerated according to decreasing absolute value. The $n$-regularized determinant of $I-K$, where $I$ denotes the identity operator on $\hil$, is
\begin{equation}
  \label{eq:18}
{\det}_n(I-K) := \left\{
  \begin{array}{cl}
 \prod_{k \in \mn} (1-\lambda_k),     & \text{if } n=1 \\[4pt]
\prod_{k \in \mn} \left[(1-\lambda_k) \exp \left( \sum_{j=1}^{n-1} \frac{\lambda_k^j}{j} \right)\right], &  \text{if } n \geq 2.
  \end{array}\right.
\end{equation}
Here the convergence of the  products on the right-hand side follows from (\ref{eq:16}).

It is clear from the definition that $I-K$ is invertible if and only if $\det_n(I-K)\neq 0$. Moreover, $\det_n(I)=1$. Since the nonzero eigenvalues of $K_1K_2$ and $K_2K_1$ coincide ($K_1,K_2 \in \bdd(\hil)$) we have
\begin{equation}\label{eq:comm}
{\det}_n(I-K_1K_2)={\det}_n(I-K_2K_1)
\end{equation}
if both $K_1K_2, K_2K_1 \in \mathcal{S}_n(\hil)$.

The regularized determinant ${\det_n}(I-K)$ is a continuous function of $K$. If $\Omega \subset \hat{\mc}$ is open and  $K(\lambda)\in \mathcal{S}_n(\hil)$ depends holomorphically on $\lambda \in \Omega$,  then $\det_n(I-K(\lambda))$ is holomorphic on $\Omega$. For a proof of both results we refer to \cite{Simon77}.

We can define the perturbation determinant for non-integer valued Schatten classes as well: Since $\mathcal{S}_p(\hil) \subset \mathcal{S}_{\lceil p \rceil}(\hil)$ where
$\lceil p \rceil = \min \{ n \in \mn :  n \geq p \}$, the $\lceil p \rceil$-regularized determinant of $I-K, K \in \mathcal{S}_p(\hil),$ is well defined, and so the above results can still be applied. Moreover, this determinant can be estimated in terms of the $p$th Schatten norm of $K$ (see \cite{b_Dunford63}, \cite{Simon77}, \cite{gil08} ): If $K\in \mathcal{S}_p(\hil)$, where $0<p<\infty$, then
        \begin{equation}\label{inequality}
          |{\det}_{\lceil p \rceil}(I-K)|\leq \exp \left( \Gamma_p \|K
        \|_{\mathcal{S}_p}^p\right),
        \end{equation}
        where $\Gamma_p $ is some positive constant.

\subsection{Perturbation theory}

The aim of perturbation theory is to obtain information about the spectrum of some operator $Z$ by showing that it is close, in a suitable sense, to an operator $Z_0$ whose spectrum is already known. In this case one can hope that some of the spectral characteristics of $Z_0$ are inherited by $Z$.  For instance, the classical Weyl theorem (see Theorem \ref{thm:2} below) implies the validity of the following result (also sometimes called Weyl's Theorem).
\begin{prop}\label{prop:3}
Let $Z, Z_0 \in \cld(\hil)$ with $\rho(Z)\cap \rho(Z_0) \neq \emptyset$. If the resolvent difference $R_{Z}(a)-R_{Z_0}(a)$ is compact for some $a \in \rho(Z) \cap \rho(Z_0)$, then $\sigma_{ess}(Z)=\sigma_{ess}(Z_0)$.
\end{prop}
\begin{rem}
If $R_{Z}(a)-R_{Z_0}(a)$ is compact for some $a \in \rho(Z) \cap \rho(Z_0)$, then the same is true for every $a \in \rho(Z) \cap \rho(Z_0)$. This is a consequence of the Hilbert-identity
  \begin{equation*}
R_{Z}(b)-R_{Z_0}(b) = (a-Z)R_{Z}(b)(R_{Z}(a)-R_{Z_0}(a))(a-Z_0)R_{Z_0}(b),
  \end{equation*}
valid for $a,b \in \rho(Z) \cap \rho(Z_0)$.
\end{rem}
Combining Proposition \ref{prop:3} and Corollary \ref{cor:a1} we obtain the following result for perturbations of selfadjoint operators.
\begin{cor}\label{prop:4}
  Let $Z, Z_0 \in \cld(\hil)$ and let $Z_0$ be selfadjoint. Suppose that there are points of $\rho(Z)$ in both the upper and lower half-planes. If $R_{Z}(a)-R_{Z_0}(a) \in \mathcal{S}_\infty(\hil)$ for some $a \in \rho(Z) \cap \rho(Z_0)$, then $\sigma_{ess}(Z)=\sigma_{ess}(Z_0) \subset \mr$ and
\begin{equation}
\label{eq:9}
  \sigma(Z)= \sigma_{ess}(Z_0) \dotcup \sigma_{d}(Z).
\end{equation}
\end{cor}

In the following we will study perturbations of the form $Z=Z_0+M$, understood as the usual operator sum defined on $\dom(Z_0) \cap \dom(M)$. More precisely, we assume that $Z_0 \in \cld(\hil)$ has non-empty resolvent set and that $M$ is a relatively bounded perturbation of $Z_0$, i.e.  $\dom(Z_0) \subset \dom(M)$ and there exist $r,s \geq 0$ such that
$$ \|Mf\| \leq r \|f\| + s \|Z_0f\|$$
for all $f \in \dom(Z_0)$. The infimum of all constants $s$ for which a corresponding $r$ exists such that the last inequality holds is called the $Z_0$-{bound} of $M$. The operator $Z$ is closed if the $Z_0$-bound of $M$ is smaller one. Note that $M$ is $Z_0$-bounded if and only if $\dom(Z_0) \subset \dom(M)$ and $MR_{Z_0}(a) \in \bdd(\hil)$ for some $a \in \rho(Z_0)$, and the $Z_0$-bound is not larger than $\inf_{a \in \rho(Z_0)} \|MR_{Z_0}(a)\|$.
The operator $M$ is called $Z_0$-{compact} if $\dom(Z_0) \subset \dom(M)$  and  $MR_{Z_0}(a) \in \mathcal{S}_\infty(\hil)$ for some $a \in \rho(Z_0)$.  Every $Z_0$-compact operator is $Z_0$-bounded and the corresponding $Z_0$-bound is $0$. Moreover, if $M$ is $Z_0$-compact and $Z_0$ is Fredholm, then also $Z_0+M$ is Fredholm (see, e.g., \cite{b_Kato95}, p.238). The last implication is the main ingredient in the proof of Weyl's theorem:

\begin{thm}\label{thm:2}
  Let $Z=Z_0+M$ where $Z_0 \in \cld(\hil)$ and $M$ is $Z_0$-compact. Then $\sigma_{ess}(Z)=\sigma_{ess}(Z_0)$.
\end{thm}
\begin{rem}
As noted above, Weyl's theorem and Proposition \ref{prop:1} 
show the validity of Proposition \ref{prop:3}.
\end{rem}


If $Z_0$ is selfadjoint and $M$ is $Z_0$-compact, then $\rho(Z)$ has values in the upper and lower half-plane (see \cite{b_Davies07}, p.326). Moreover, if $a \in \rho(Z) \cap \rho(Z_0)$, then $R_Z(a)-R_{Z_0}(a) \in \mathcal{S}_\infty(\hil)$ as a consequence of the {second resolvent identity}
  \begin{equation}
    \label{eq:10}
    R_Z(a)-R_{Z_0}(a)=R_Z(a)MR_{Z_0}(a).
  \end{equation}
So Corollary \ref{prop:4} implies that $\sigma_{ess}(Z)=\sigma_{ess}(Z_0)$ and $\sigma(Z)= \sigma_{ess}(Z_0) \dotcup \sigma_d(Z)$.\\

\subsection{Perturbation determinants}\label{sec:pert-determ}

We have seen in the last section that the essential spectrum is stable under (relatively) compact perturbations. In this section, we will have a look at the discrete spectrum and construct a holomorphic function whose zeros coincide with the discrete eigenvalues of the corresponding operator. Throughout we make the following assumption.
\begin{assumption}\label{ass:1}
$Z_0$ and $Z$ are closed densely defined operators in $\hil$ such that
  \begin{enumerate}
  \item[(i)] $\rho(Z_0) \cap \rho(Z) \neq \emptyset$.

  \item[(ii)] $R_{Z}(b)-R_{Z_0}(b) \in \mathcal{S}_p(\hil)$ for some $b \in \rho(Z_0) \cap \rho(Z)$ and some fixed $p>0$.

\item[(iii)] $\sigma(Z) \cap \rho(Z_0) = \sigma_d(Z)$.
  \end{enumerate}
\end{assumption}

\begin{rem}\label{rem:6}
  By Proposition \ref{prop:4}, assumption (iii) follows from assumption (ii) if $Z_0$ is selfadjoint with $\sigma_d(Z_0)=\emptyset$ and if there exist points of $\rho(Z)$ in both the upper and lower half-planes. If $Z_0$ and $Z$ are bounded operators on $\hil$ then  the second resolvent identity implies that assumption (ii) is equivalent to $Z-Z_0 \in \mathcal{S}_p(\hil)$.
\end{rem}


We begin with the case when $Z_0,Z \in \bdd(\hil)$: Then for $\lambda_0 \in \rho(Z_0)$ we have
$$(\lambda_0 - Z)R_{Z_0}(\lambda_0) = I-(Z-Z_0)R_{Z_0}(\lambda_0),$$
so $\lambda_0 \in \rho(Z)$ if and only if $I-(Z-Z_0)R_{Z_0}(\lambda_0)$ is invertible. As we know from Section \ref{sec:schatten}, this operator is invertible if and only if $$\detp(I-(Z-Z_0)R_{Z_0}(\lambda_0))\neq 0.$$

By Assumption \ref{ass:1} we have  $\sigma(Z) \cap \rho(Z_0)=\sigma_d(Z)$, so we have shown that  $\lambda_0 \in \sigma_d(Z)$ if and only if $\lambda_0$ is a zero of the analytic function
\begin{equation}\label{eq:191}
d_\infty^{Z,Z_0} : \hat{\rho}(Z_0) \to \mc, \qquad d_\infty^{Z,Z_0}(\lambda) := \detp(I-(Z-Z_0)R_{Z_0}(\lambda)).
\end{equation}
 For later purposes we note that $d_\infty^{Z,Z_0}(\infty)=1$.\\

\noindent Next, we consider the general case: Let $a \in \rho(Z_0) \cap \rho(Z)$ where $Z_0, Z$ satisfy Assumption \ref{ass:1}. Then Proposition \ref{prop:1} and its accompanying remark show that
$$\sigma_d(R_Z(a))= \sigma(R_{Z}(a)) \cap \rho(R_{Z_0}(a)),$$
so we can apply the previous discussion to the operators  $R_{Z_0}(a)$ and $R_{Z}(a)$, i.e. the function
\begin{equation}\label{eq:189}
d_\infty^{R_Z(a),R_{Z_0}(a)}(.)= \detp(I-[R_Z(a)-R_{Z_0}(a)][(.)-R_{Z_0}(a)]^{-1})
\end{equation}
 is well defined and analytic on $\hat{\rho}(R_{Z_0}(a))$. Moreover, since $\lambda \in \hat{\rho}(Z_0)$
 if and only if $(a-\lambda)^{-1} \in \hat{\rho}(R_{Z_0}(a))$  (which is again a consequence of Proposition \ref{prop:1} and Remark \ref{rem:3}), we see that the function
\begin{equation}\label{eq:190}
   d_a^{Z,Z_0}(\lambda):=d_\infty^{R_{Z}(a),R_{Z_0}(a)}((a-\lambda)^{-1})
\end{equation}
is analytic on $\hat{\rho}(Z_0)$ and
\begin{equation*}
d_a^{Z,Z_0}(\lambda)=0 \quad \Leftrightarrow \quad   (a-\lambda)^{-1} \in \sigma_d(R_Z(a)) \quad \Leftrightarrow \quad \lambda \in \sigma_d(Z).
\end{equation*}
Note that, as above, we have $d_a^{Z,Z_0}(a)= d_\infty^{R_{Z}(a),R_{Z_0}(a)}(\infty)=1$.\\

We summarize the previous discussion in the following proposition.
\begin{prop}
Let $a \in \hat{\rho}(Z_0) \cap \hat{\rho}(Z)$, where $Z,Z_0$ satisfy Assumption \ref{ass:1}, and let $d_a=d_a^{Z,Z_0} : \hat{\rho}(Z_0) \to \mc$ be defined by (\ref{eq:191}) if $a=\infty$ and by (\ref{eq:190}) if $a \neq \infty$, respectively. Then $d_a$ is analytic, $d_a(a)=1$ and $\lambda \in \sigma_d(Z)$ if and only if $d_a(\lambda)=0$.
\end{prop}

We call the function $d_a=d_a^{Z,Z_0}$ the $p$th perturbation determinant of $Z$ by $Z_0$ (the $p$-dependence of $d_a$ is neglected in our notation). Without proof we note that the algebraic multiplicity of $\lambda_0 \in \sigma_d(Z)$ coincides with the order of $\lambda_0$ as a zero of $d_a$, see \cite{Hans_diss}, p.20-22.
\begin{rem}
Our definition of perturbation determinants is an extension of the standard one (which coincides with the function $d_\infty$), see, e.g., \cite{b_Gohberg69} and \cite{b_Yafaev92}.
\end{rem}
We conclude this section with some estimates.
\begin{prop}\label{prop:12}
Let $a \in {\rho}(Z_0) \cap {\rho}(Z)$, where $Z,Z_0$ satisfy Assumption \ref{ass:1}, and let $d_a: \hat{\rho}(Z_0) \to \mc$ be defined as above. Then, for $\lambda\neq a$,
  \begin{equation}
    \label{eq:29}
    |d_a(\lambda)| \leq \exp \left( \Gamma_p \|[R_Z(a)-R_{Z_0}(a)][(a-\lambda)^{-1}-R_{Z_0}(a)]^{-1}\|_{\mathcal{S}_p}^p\right),
  \end{equation}
where $\Gamma_p$ was introduced in estimate (\ref{inequality}).
\end{prop}
\begin{proof}
  Apply estimate (\ref{inequality}).
\end{proof}
\begin{prop}\label{prop:17}
Let $Z,Z_0 \in \bdd(\hil)$  satisfy Assumption \ref{ass:1}. Then for $\lambda \in \hat{\rho}(Z_0)$ we have
\begin{equation}
\label{eq:209}
    |d_\infty(\lambda)| \leq \exp\left(\Gamma_p \|(Z-Z_0)R_{Z_0}(\lambda)\|_{\mathcal{S}_p}^p\right).
\end{equation}
If, in addition, $Z-Z_0=M_1M_2$ where $M_1,M_2$ are bounded operators on $\hil$ such that $M_2 R_{Z_0}(a)M_1 \in \mathcal{S}_p(\hil)$ for every $a \in \rho(Z_0)$, then for $\lambda \in \hat{\rho}(Z_0)$ we have
\begin{equation}
  \label{eq:39}
    |d_\infty(\lambda)| \leq \exp\left(\Gamma_p \|M_2R_{Z_0}(\lambda)M_1\|_{\mathcal{S}_p}^p\right).
\end{equation}
\end{prop}
\begin{proof}
 Estimate (\ref{eq:39}) is a consequence of estimate (\ref{inequality}), the definition of $d_\infty$ and the identity
$$\detp(I-(Z-Z_0)R_{Z_0}(\lambda))=\detp(I-M_1M_2R_{Z_0}(\lambda))=\detp(I-M_2R_{Z_0}(\lambda)M_1),$$
which follows from (\ref{eq:comm}). Estimate (\ref{eq:209}) follows immediately from the definition of $d_\infty$ and  estimate (\ref{inequality}).
\end{proof}
\begin{rem}
While the non-zero eigenvalues of $M_1M_2R_{Z_0}(a)$ and $M_2R_{Z_0}(a)M_1$ coincide, the same need not be true for their singular values. In particular, while \linebreak $(Z-Z_0)R_{Z_0}(a) \in \mathcal{S}_p(\hil)$ is automatically satisfied if $Z,Z_0 \in \bdd(\hil)$ satisfy Assumption \ref{ass:1}, in general this need not imply that $M_2R_{Z_0}(a)M_1 \in \mathcal{S}_p(\hil)$ as well.
\end{rem}

\section{Zeros of holomorphic functions}\label{cha:complex}

In this chapter we discuss results on the distribution of zeros of holomorphic functions, which will
subsequently be applied to the holomorphic functions defined by perturbation determinants to obtain results on the
distribution of eigenvalues for certain classes of operators. We begin with a motivating discussion in Section \ref{sec:short-motivation}, introducing
the class of functions on the unit disk which will be our special focus of study. In Section \ref{sec:jensen} we consider results that can be
obtained using the classical Jensen identity. In Section \ref{sec:c4} we present the recent results of Borichev, Golinskii and Kupin and show that, for the
class of functions that we are interested in, they yield more information than provided by the application of the Jensen identity.

\subsection{Motivation: the complex analysis method for studying eigenvalues}\label{sec:short-motivation}

We have seen in Section \ref{sec:pert-determ} that the discrete spectrum of a linear operator $Z$ satisfying Assumption \ref{ass:1} coincides with the zero set of the corresponding perturbation determinant, which is a holomorphic function defined on the resolvent set of the `unperturbed' operator $Z_0$. Moreover, we have a bound
on the absolute value of this holomorphic function in the form of Propositions \ref{prop:12} and \ref{prop:17}. Thus, general results providing information about
the zeros of holomorphic functions satisfying certain bounds may be exploited to obtain information about the eigenvalues of the operator $Z$. This
observation is the basis of the following complex-analysis approach to studying eigenvalues.

As an example, we consider the following situation: $Z_0 \in \bdd(\hil)$ is assumed to be a selfadjoint operator with
\begin{equation}
  \label{eq:30}
  \sigma(Z_0)= \sigma_{ess}(Z_0)= [a, b],
\end{equation}
where $a < b $, and
$$Z=Z_0+M,$$
where $M \in \mathcal{S}_p(\hil)$ for some fixed $p>0$. Given these assumptions, the spectrum of $Z$ can differ from the spectrum of $Z_0$ by an at most countable set of discrete eigenvalues, whose points of accumulation are contained in the interval $[a, b]$.
Moreover, $\sigma_d(Z)$ is precisely the zero set of the $p$th perturbation determinant $d=d_\infty^{Z,Z_0}$ defined by
$$d : \hat{\mc}\setminus [a,b] \to \mc, \qquad d(\lambda) = \detp(I-MR_{Z_0}(\lambda)).$$
It should therefore be possible to obtain further information on the distribution of the eigenvalues of $Z$ by studying the analytic function $d$, in particular, by taking advantage of the estimate provided on $d$ in Proposition \ref{prop:17}, i.e.,
\begin{equation}\label{eq:31}
 \log |d(\lambda)| \leq  \Gamma_p \|MR_{Z_0}(\lambda)\|_{\mathcal{S}_p}^p, \qquad \lambda \in \mc\setminus [a,b],
\end{equation}
as well as the fact that $d(\infty)=1$.
Note that the right-hand side of (\ref{eq:31}) is finite for any $\lambda\in \hat{\mc}\setminus [a,b]$, but as $\lambda$ approaches $[a,b]$
it can `explode'. A simple way to estimate the right-hand side of (\ref{eq:31}) from above and thus to obtain a more concrete estimate, is to use the identity
\begin{equation}
  \label{eq:32}
 \|R_{Z_0}(\lambda)\|=[\dist(\lambda,\sigma(Z_0))]^{-1},
\end{equation}
which is valid since $Z_0$ is selfadjoint, and the inequality (\ref{eq:14}) to obtain
\begin{equation}\label{eq:33}
\log |d(\lambda)| \leq   \frac{\Gamma_p \|M\|_{\mathcal{S}_p}^p}{\dist(\lambda,\sigma(Z_0))^p} .
\end{equation}
The inequality (\ref{eq:33}) is the best that we can obtain at a general level, that is without imposing any further restrictions on the
operators $Z_0$ and $M$. However, as we shall show in Chapter \ref{sec:jacobi}, for concrete operators it is possible to obtain better inequalities by a more
precise analysis of the $\mathcal{S}_p$-norm of $MR_{Z_0}(\lambda)$. These inequalities will take the general form
\begin{equation}
  \label{eq:34}
 \log |d(\lambda)|\leq  \frac{C}{\dist(\lambda,\sigma(Z_0))^{\alpha'} \dist(\lambda,a)^{\beta_1'}\dist(\lambda,b)^{\beta_2'}},
\end{equation}
where  $\alpha'$ and $\beta_1',\beta_2'$ are some non-negative parameters with $\alpha' + \beta_1'+\beta_2' = p$. Note that (\ref{eq:34}) can be stronger than (\ref{eq:33}) in the sense that
the growth of $\log|d(\lambda)|$ as $\lambda$ approaches a point $\zeta\in (a,b)$ is estimated from above by $O(|\lambda-\zeta|^{-\alpha'})$, which can be smaller than the $O(|\lambda -\zeta|^{-p})$ bound given by (\ref{eq:33}) (since $\alpha'< p$ if $\beta_1'+\beta_2' > 0$). A similar remark applies to $\lambda$ approaching one of the endpoints $a,b$ (since, e.g., $\alpha'+\beta_1'<p$ if $\beta_2'>0$).
As we shall see, such differences are very significant in terms of the estimates on eigenvalues that are obtained.

The question then becomes how to use inequalities of the type (\ref{eq:33}), (\ref{eq:34}) to deduce information about the zeros of the holomorphic function
$d(\lambda)$. The study of zeros of holomorphic functions is, of course, a major theme in complex analysis. Since the holomorphic functions
$d(\lambda)$ which we will be looking at will be defined on domains that are conformally equivalent to the open unit disk $\md$, we
are specifically interested in results about zeros of functions $h\in H(\md)$, the class of holomorphic functions in the unit disk. Indeed, if $\Omega\subset {\hat{\mc}}$ is a domain
which is conformally equivalent to the unit disk, we choose a conformal map $\phi:\md \rightarrow \Omega$ so that the study of the zeros of the holomorphic function
$d:\Omega\rightarrow \mc$ is converted to the study of the zeros of the function $h=d\circ \phi:\md \rightarrow \mc$, where, denoting by $\mathcal{Z}(h)$
the set of zeros of a holomorphic function $h$, we have
$$\mathcal{Z}(d|_{\Omega})=\phi(\mathcal{Z}(h)).$$
We can also choose the conformal mapping $\phi$ so that $\phi(0)=\infty$, which implies that $h(0)=1$.

This conversion involves
two steps which require some effort:

(i) Inequalities of the type (\ref{eq:33}) and (\ref{eq:34}) must be translated into inequalities on the
function $h\in H(\md)$.
(ii) Results obtained about the zeros of $h$, lying in the unit disk, must be translated into results about the zeros of $d$.

Regarding step (i), it turns out that inequalities of the form (\ref{eq:34}), and generalizations of it, are converted into inequalities of the form
\begin{equation}\label{eq:c03}
 \log |h(w)| \leq  \frac {K |w|^\gamma } {(1-|w|)^\alpha \prod_{j=1}^N|w-\xi_j|^{\beta_j}}, \quad w \in \md,
\end{equation}
where $\xi_j\in\mt:=\partial \md$ and the parameters in (\ref{eq:c03}) are determined by those appearing in the inequality bounding $d(\lambda)$ and by properties of the
conformal mapping $\phi$. Note that this inequality restricts the growth of $|h(w)|$ as $|w|\rightarrow 1$ differently according to whether or
not $w$ approaches one of the `special' points $\xi_j$. Since functions obeying (\ref{eq:c03}) play an important role in our work, it is convenient
to have a special notation for this class of functions. First, let us set
\begin{equation}
  \label{eq:22}
 (\mt^N)_* := \{ (\xi_1, \ldots, \xi_N) \in \mt^N : \xi_i \neq \xi_j, 1 \leq i < j \leq N\}, \quad N \in \mn.
\end{equation}
\begin{definition}\label{def:c01}
  Let $\alpha, \gamma, K \in \mr_+:=[0,\infty)$. For $N \in \mn$ let
$\vec \beta =(\beta_1, \ldots, \beta_N) \in \mr_+^N$ and $\vec \xi =(\xi_1, \ldots, \xi_N)\in (\mt^N)_*$. The class of all functions $h \in H(\md)$ satisfying $h(0)=1$ and obeying (\ref{eq:c03}) (for this choice of parameters) is denoted by $\mathcal{M}(\alpha, \vec \beta, \gamma, \vec \xi, K)$. Moreover, we set $\mathcal{M}(\alpha,K)= \mathcal{M}(\alpha, \vec 0 , 0, \vec \xi, K)$ where $\vec \xi \in \mt^N$ is arbitrary, that is functions satisfying
$$ \log |h(w)| \leq  \frac {K } {(1-|w|)^\alpha }, \quad w \in \md.$$
\end{definition}
\begin{rem}
Throughout this chapter, whenever speaking of $\mathcal{M}(\alpha, \vec \beta, \gamma, \vec \xi, K)$ we will always implicitly assume that the parameters are chosen as indicated in the previous definition.
\end{rem}
\begin{rem}\label{rem:c02}
We have the inclusions
  \begin{equation*}
    \mathcal{M}(\alpha,\vec \beta, \gamma, \vec \xi , K) \subset     \mathcal{M}(\alpha',\vec \beta, \gamma', \vec \xi , K')
  \end{equation*}
if $\alpha \leq \alpha', \gamma \geq \gamma'$ and $K \leq K'$, and
  \begin{equation*}
    \mathcal{M}(\alpha,\vec \beta, \gamma, \vec \xi , K) \subset     \mathcal{M}(\alpha,\vec \beta', \gamma, \vec \xi , K \cdot 2^{\sum_{j=1}^N \beta_j'})
  \end{equation*}
if $\beta_j \leq \beta_j'$ for $1 \leq j \leq N$.
\end{rem}


Thus, our aim is to understand what information on the set of zeros of $h$ is implied by the assumption $h\in \mathcal{M}(\alpha, \vec \beta, \gamma, \vec \xi, K)$.
This information will then be translated back into information about the set of zeros of the perturbation determinant $d(\lambda)$, that is
about the eigenvalues of $Z$.

\subsection{Zeros of holomorphic functions in the unit disk: Jensen's identity}\label{sec:jensen}

The zero set of a (non-trivial) function $h \in H(\md)$ is of course discrete, with possible accumulation points on the boundary $\mt$. In other words, $\mathcal{Z}(h)$ is either finite, or it can be written as $\mathcal{Z}(h)=\{ w_k\}_{k=1}^{\infty}$, where $|w_k|$ is increasing, and
\begin{equation}\label{eq:c01}
\lim_{k \to \infty}(1 - |w_k|) = 0.
\end{equation}
While in this generality nothing more can be said about $\mathcal{Z}(h)$, the situation changes drastically if we restrict the growth of $|h(z)|$ as $z$
approaches the boundary of the unit disk.  A basic result which allows to make a connection between the boundary growth of a function $h\in H(\md)$ and the distribution of its
zeros is Jensen's identity (see \cite{b_Rudin87}, p.308). Denoting the number of zeros (counting multiplicities) of $h$ in the disk $\md_r=\{ w \in \mc : |w| \leq r\}$ by $N(h,r)$, this result reads as follows.
\begin{lemma}\label{lem:c01}
  Let $h \in H(\md)$ with $|h(0)|=1$. Then for $r \in (0,1)$ we have
  \begin{equation}\label{eq:c07}
    \int_0^r \frac{N(h,s)}{s}\: ds = \sum_{w \in \mathcal{Z}(h), |w|\leq r} \log\left| \frac{r}{w} \right| {=}  \frac{1}{2\pi} \int_0^{2\pi} \log | h(re^{i\theta}) | \:d\theta.
  \end{equation}
\end{lemma}
Note that the left equality is immediate, while the right equality is the real content of the result.

As a simple application of Jensen's identity, consider the case in which $h \in H^\infty(\md)$, the class of functions bounded in the unit disk, with
$\|h\|_\infty$ denoting the supremum.  Then the right-hand side of (\ref{eq:c07}) is bounded from above by $\log(\|h\|_\infty)$, so that we can take the limit
$r\rightarrow 1-$ (noting that the left-hand side increases with $r$) and obtain
$$\sum_{w \in \mathcal{Z}(h)} \log\left| \frac{1}{w} \right|\leq \log(\|h\|_{\infty}).$$
We may also bound the left-hand side of this inequality from below, using $\log|w|\leq |w|-1$, to obtain
\begin{equation}\label{eq:c02}
  \sum_{w \in \mathcal{Z}(h)}  (1-|w|) \leq \log(\|h\|_{\infty}).
\end{equation}
Obviously the convergence of the sum in (\ref{eq:c02}), known as the Blaschke sum, is a much stricter condition on the sequence of zeros than
(\ref{eq:c01}).
\noindent However, the functions $h$ arising in the applications we make to the perturbation determinant will generally not be bounded, so the Blaschke condition (\ref{eq:c02}) cannot be applied. We will now assume that $h \in \mathcal{M}=\mathcal{M}(\alpha,\vec \beta, \gamma, \vec \xi, K)$ and derive estimates on
the zeros of $h$, by using Jensen's identity in a more careful way.

We will use the following proposition, derived from Jensen's identity.
For that purpose we denote the \textit{support} of a function $f: (a,b) \subset \mr \to \mr$ by $\supp(f)$, i.e.,
\begin{equation*}
  \label{eq:181}
\supp(f)= \overline{\{ x \in (a,b) : f(x) \neq 0 \}  }.
\end{equation*}
Moreover, by $f_+=\max(f,0)$ and $f_-=-\min(f,0)$ we denote the positive and negative parts of $f$, respectively (note that we will use the same notation for the positive and negative parts of a real number as well). In addition, we denote the class of all twice-differentiable functions on $(a,b)$ whose second derivative is continuous by $C^2(a,b)$.

 \begin{prop}\label{prop:c03}
  Let $\vphi \in C^2(0,1)$ be non-negative and non-increasing, and suppose that  $\lim_{r \to 1} \vphi(r) = \lim_{r \to 1} \vphi'(r)=0$, $  \supp \left( [ r\vphi'(r) ]' \right)_- \subset [0,1)$ and $$\sup_{0<r<1} \left( [r\vphi'(r)]' \right)_-<\infty.$$ If $h \in H(\md)$, with $|h(0)|=1$, then
\begin{equation} \label{eq:c21}
  \sum_{w \in \mathcal{Z}(h)} \vphi(|w|) = \frac 1 {2\pi} \int_0^1 dr \:
  [ r\vphi'(r) ]'   \int_0^{2\pi} d\theta \log |h(re^{i\theta})|.\footnote{Of course, both sides of (\ref{eq:c21}) may be (simultaneously) divergent.}
\end{equation}
\end{prop}
\begin{rem}
We are mainly interested in the choice $\vphi(r)=(1-r)^q$, with $q > 1$; other possible choices are $\vphi(r)=(-\log(r))^q$ and $\vphi(r)=(r^{-1}-r)^q$, respectively.
\end{rem}
\begin{proof}[Proof of Proposition \ref{prop:c03}]
Let $0<r<1$. We restate Jensen's identity:
\begin{equation}\label{eq:c22}
\int_0^r ds \: \frac{N(h,s)}{s} = \frac{1}{2\pi}\int_0^{2\pi}\: d\theta  \log| h(re^{i\theta})|.
\end{equation}
Multiplying both sides of (\ref{eq:c22}) by $[
r\vphi'(r) ]' $ and integrating with respect to $r$ leads to
 \begin{eqnarray}
&&  \frac{1}{2\pi} \int_0^1 dr\: [ r\vphi'(r) ]'  \int_0^{2\pi} d \theta\: \log | h(re^{i\theta})|  \nonumber \\
&=& \int_0^1 dr \: [ r\vphi'(r) ]'  \int_0^r ds \: \frac{N(h,s)}{s}
\overset{(\star)}{=} \int_0^1 ds \:  \frac{N(h,s)}{s} \int_s^1 dr \:[ r\vphi'(r) ]' \nonumber \\
&=& - \int_0^1 ds \: \vphi'(s)  N(h,s) = \int_0^\infty dt \:  \left[ \frac d {dt} \vphi(e^{-t}) \right]
N(h,e^{-t}). \label{eq:c23}
 \end{eqnarray}
The application of Fubini's theorem in ($\star$) is justified by the assumptions made on $\vphi$.
We can reformulate the right-hand side of the last equation as follows
\begin{eqnarray}
&& \int_0^\infty dt \: \left[ \frac d {dt} \vphi(e^{-t}) \right] N(h,{e^{-t}}) = \int_0^\infty dt \: \sum_{w \in \mathcal{Z}(h), |w| < e^{-t}} \left[ \frac d {dt} \vphi(e^{-t}) \right] \nonumber \\
&=& \sum_{w \in \mathcal{Z}(h)} \int_0^{-\log|w|} \: dt \: \left[ \frac
d {dt} \vphi(e^{-t}) \right] = \sum_{w \in \mathcal{Z}(h)} \vphi(|w|).
\nonumber
\end{eqnarray}
The last equation together with (\ref{eq:c23}) yields the result.
\end{proof}
We can now derive a Blaschke-type result on the zeros of a function $h \in \mathcal{M}$ (see Definition \ref{def:c01}). In the result below, $C(\alpha,\vec \beta, \vec \xi,\tau)$ denotes a
constant depending only on the parameters $\alpha,\vec \beta, \vec \xi,\tau$, which can in principle be made explicit but would yield expressions
too unwieldy to be of much use. As usual, when such a constant appears in two equations, or even on two lines of the same equations, it may take
different values, but we do take care to always indicate the parameters on which the constant depends.

\begin{thm}\label{thm:c03}
Let $h \in \mathcal{M}(\alpha, \vec \beta, 0, \vec \xi, K)$. Then for every $\tau > 0$ we have
\begin{equation}
  \label{eq:c24}
  \sum_{w \in \mathcal{Z}(h)} (1-|w|)^{1 + \alpha  + \max_j (\beta_j-1)_+ + \tau} \leq C(\alpha,\vec \beta, \vec \xi,\tau) K.
\end{equation}
\end{thm}
\begin{proof}
For $q > 1$ let $\vphi(r)=(1-r)^q$. Since
\[ [ r\vphi'(r) ]'  = q  (1-r)^{q-2} (rq -1)\] we obtain from Proposition \ref{prop:c03} and our assumptions, using that $\int_0^{2\pi}\log |h(re^{i\theta})| d\theta$ is non-negative,
 \begin{eqnarray}
&&   \sum_{w \in \mathcal{Z}(h)} (1-|w|)^q = \frac {q} {2\pi} \int_0^1 dr \frac{(rq-1)}{(1-r)^{2-q}} \int_0^{2\pi} d\theta \log |h(re^{i\theta})| \nonumber \\
&\leq& \frac {q}{2\pi} \int_{1/{q}}^1 dr \frac{(rq-1)}{(1-r)^{2-q}}   \int_0^{2\pi} {d\theta} \: \log |h(re^{i\theta})|  \nonumber \\
&\leq& \frac {K q(q-1)}
{2\pi} \int_{1/{q}}^1 dr \frac{1}{(1-r)^{2-q+\alpha}}
\int_0^{2\pi} \frac{d\theta}{\prod_{j=1}^N |re^{i\theta}-\xi_j|^{\beta_j}} \nonumber \\
&\leq&  \frac {K C(\vec \beta, \vec \xi)  q(q-1)}
{2\pi} \sum_{j=1}^N \int_{1/{q}}^1 dr \frac{1}{(1-r)^{2-q+\alpha}}
\int_0^{2\pi} \frac{d\theta}{|re^{i\theta}-\xi_j|^{\beta_j}}.\label{eq:c26}
\end{eqnarray}
Standard calculations show that, as $r\rightarrow 1-$
\begin{equation}\label{eq:c12}
 \int_0^{2\pi} \frac{d\theta}{|re^{i\theta}-\xi|^\beta} = \left\{
   \begin{array}{cl}
     O\left( \frac{1}{(1-r)^{\beta-1}} \right), & \text{ if } \beta > 1, \\[6pt]
     O\left( -\log(1-r) \right), & \text{ if }\beta = 1, \\[6pt]
     O\left(1\right), & \text{ if } \beta < 1.
   \end{array}\right.
\end{equation}
Therefore integrals on the right-hand side of (\ref{eq:c26}) will be finite whenever $q > 1 + \alpha + \max_j (\beta_j-1)_+$, and the result follows.
\end{proof}

 \subsection{A theorem of Borichev, Golinskii and Kupin}\label{sec:c4}

A different inequality on the zeros of $h\in \mathcal{M}(\alpha,\vec \beta, 0, \vec \xi, K)$ was proved by  Borichev, Golinskii and Kupin \cite{Borichev08}.
\begin{thm}\label{thm:c05}
Let $h \in \mathcal{M}(\alpha, \vec \beta,0,\vec \xi, K)$, where $\vec \xi = (\xi_1, \ldots,$ $\xi_N) \in (\mt^N)_*$ and \linebreak $\vec \beta = ( \beta_1, \ldots, \beta_N) \in \mr_+^N$. Then for every $\tau > 0$ the following holds: If $\alpha > 0$ then
\begin{equation}\label{eq:c35}\sum_{w \in \mathcal{Z}(h)}(1-|w|)^{\alpha+1+\tau}\prod_{j=1}^N|w-\xi_j|^{(\beta_j-1+\tau)_+}\leq
C(\alpha,\vec \beta,\vec \xi, \tau) K.\end{equation}
Furthermore, if $\alpha = 0$ then
\begin{equation}\label{eq:c36}\sum_{w \in \mathcal{Z}(h)}(1-|w|)\prod_{j=1}^N|w-\xi_j|^{(\beta_j-1+\tau)_+}\leq
C(\vec \beta,\vec \xi, \tau) K.\end{equation}
\end{thm}
To see the advantage of (\ref{eq:c35}) over (\ref{eq:c24}), consider a convergent subsequence $\{w_k\}_{k=1}^\infty \subset \mathcal{Z}(h)$. The limit point $\xi$
satisfies  $|\xi|=1$, and (\ref{eq:c24}) ensures that the sum
\begin{equation}\label{eq:sum}\sum_{k=1}^\infty (1-|w_k|)^\eta<\infty \end{equation}
whenever \begin{equation}\label{eq:condition}\eta>1 + \alpha  + \max_j (\beta_j-1)_+ .\end{equation}
As for (\ref{eq:c35}), it gives us different information according to whether $\xi=\xi_j$ for some $1\leq j\leq N$ or whether $\xi\in \partial \md$ is a
`generic' point. For sequences $\{w_k\}_{k=1}^\infty$ converging to generic points ($\xi\neq \xi_j$) the product term in (\ref{eq:c35}) will be bounded from below by a positive constant along the sequence, so that we can conclude that (\ref{eq:sum}) will hold whenever $\eta>\alpha+1$, obviously a less restrictive condition than that
provided by (\ref{eq:condition}), except in the case when $\beta_j\leq 1$ for all $j$, in which the two conditions are the same.
When $\xi=\xi_{j^*}$ for some $1\leq j^*\leq N$, the summands in (\ref{eq:c35}) will be bounded from below by a positive constant multiple of
$$(1-|w_k|)^{\alpha+1+\tau}|w_k-\xi_{j^*}|^{(\beta_{j^*}-1+\tau)_+}\geq (1-|w_k|)^{\alpha+1+\tau+(\beta_{j^*}-1+\tau)_+}.$$
Therefore, if $\beta_{j^*}>1$, (\ref{eq:c35}) implies that (\ref{eq:sum}) will hold whenever $\eta>\alpha+1+\beta_{j^*}$, a less restrictive condition than (\ref{eq:condition}) since it does not involve the maximum of all $\beta_{j}$'s.  If $\beta_{j^*}<1$, (\ref{eq:c35}) implies that (\ref{eq:sum}) will hold whenever $\eta>\alpha+1$, also a less restrictive condition except in the case where all $\beta_j<1$ for all $1\leq j\leq N$, in which it is the same condition.

We thus see that the Theorem of Borichev, Golinskii and Kupin \cite{Borichev08} provides sharper information about the asymptotic distribution
of the zeros than Theorem \ref{thm:c03}. Therefore, in our applications, Theorem \ref{thm:c05} will provide more precise information about the
distribution of eigenvalues, and it is this result which will be used. It should be noted that, unlike Theorem \ref{thm:c03}, the proof  of Theorem \ref{thm:c05}
is not an application of Jensen's identity, and requires less elementary function-theoretic arguments.
\begin{rem}
  We should also note that Theorem \ref{thm:c05} has been generalized in several ways: to subharmonic functions on the unit disk \cite{Golinskii08}, and to holomorphic functions on more general domains \cite{golinskii11}, \cite{golinskii12}. We will return to this topic in Chapter \ref{sec:outlook}.
\end{rem}

In the following, however, we will make one improvement to Theorem \ref{thm:c05}, which is useful when considering applications to eigenvalue estimates.
We consider functions $h \in \mathcal{M}(\alpha, \vec \beta, \gamma, \vec \xi, K)$ with $\gamma > 0$, which have the property that
$$\log|h(w)|=O(|w|^\gamma), \quad \text{as } |w| \to 0.$$
 Of course these functions are included
in $\mathcal{M}(\alpha, \vec \beta, 0, \vec \xi, K)$, so that Theorem \ref{thm:c05} holds for them.
We will show that, for this class of functions, the sum on the left-hand side of (\ref{eq:c35}) can be replaced by
\begin{equation}\label{eq:c27}
  \sum_{w\in \mathcal{Z}(h)} \frac{(1-|w|)^{\alpha+1+\tau}}{|w|^{x}} \prod_{j=1}^N|w-\xi_j|^{(\beta_j-1+\tau)_+}
\end{equation}
for a suitable choice of $x=x(\gamma)>0$. It should be noted that since we always assume $h(0)=1$, the zeros of $h$ will always be bounded away from $0$,
so that (\ref{eq:c35}) implies that the sum (\ref{eq:c27}) is finite. The point, however, is to obtain a bound on this sum which is linear in $K$, like
the bound in Theorem \ref{thm:c05}. This linearity is important in the applications.

We first estimate the counting function $N(h,r)$ for small $r>0$.
\begin{lemma}\label{lem:c03}
Let $h \in \mathcal{M}(\alpha, \vec \beta, \gamma, \vec \xi, K)$. Then for $r \in (0,\frac 1 2]$ we have
\begin{equation}\label{eq:c28}
  N(h,r) \leq C(\alpha,\vec \beta, \vec \xi) K r^{\gamma}  .
\end{equation}
\end{lemma}
\begin{proof}
Let $0 < r < s < 1$. Then,
\begin{eqnarray*}
N(h,r) = \frac{1}{\log (\frac s r) } \int_r^s \frac{N(h,r)}{t} dt
\leq \frac{1}{\log (\frac s r) } \int_{r}^s \frac{N(h,t)}{t} dt
\leq \frac{1}{\log (\frac s r) } \int_{0}^s \frac{N(h,t)}{t} dt.
\end{eqnarray*}
Jensen's identity and our assumptions on $h$ thus imply that
\begin{eqnarray*}
N(h,r)  &\leq& \frac{1}{\log (\frac s r) } \frac{1}{2\pi} \int_0^{2\pi} \log |h(s e^{i\theta})| d\theta
\leq \frac{1}{\log (\frac s r) }  \frac{K
s^\gamma}{(1-s)^\alpha} \frac{1}{2\pi} \int_0^{2\pi}
\prod_{j=1}^N \frac{1}{ | s e^{i\theta} - \xi_j|^{\beta_j}} d\theta .
\end{eqnarray*}
Choosing $s = \frac{3}{2} r$ (i.e., $s \leq \frac 3 4$) concludes the proof.
\end{proof}
The information offered by the previous lemma can immediately be applied to obtain the following result.
\begin{lemma}\label{lem:c04}
  Let $h \in \mathcal{M}(\alpha, \vec \beta, \gamma, \vec \xi, K)$. Then for every $\eps >0$ we have
  \begin{equation}
    \label{eq:c29}
    \sum_{w \in \mathcal{Z}(h), |w| \leq \frac 1 2} \frac{1}{|w|^{(\gamma-\eps)_+}} \leq C(\alpha, \vec \beta, \gamma, \vec \xi, \eps) K.
  \end{equation}
\end{lemma}
\begin{proof}
For $\gamma \leq \eps$ the left-hand side of (\ref{eq:c29}) is equal to $N(h,1/2)$, so in view of Lemma \ref{lem:c03} we only need to consider the case $\gamma > \eps$. In this case, we can rewrite the sum in (\ref{eq:c29}) as follows:
\begin{eqnarray*}
\sum_{w \in \mathcal{Z}(h), |w| \leq \frac 1 2} \frac{1}{|w|^{\gamma-\eps}}
&=& (\gamma-\eps) \sum_{w \in \mathcal{Z}(h), |w| \leq \frac 1 2} \int_0^{\frac 1 {|w|}} dt \; t^{\gamma-1-\eps} \\
&=& (\gamma-\eps) \left[ \int_0^2 dt\; t^{\gamma-1-\eps} N(h,1/2) +  \int_2^\infty dt\; t^{\gamma-1-\eps} N(h,t^{-1})  \right].
\end{eqnarray*}
Using Lemma \ref{lem:c03} and the fact that $\gamma > \eps$ we conclude that
\begin{eqnarray*}
  \int_0^2 dt \; t^{\gamma-1-\eps} N(h,1/2)
&\leq& C(\alpha,\vec \beta, \gamma, \vec \xi, \eps) K.
\end{eqnarray*}
Similarly, using that $\eps>0$, Lemma \ref{lem:c03} implies that
\begin{eqnarray*}
\int_2^\infty dt \; t^{\gamma-1-\eps} N(h,{t^{-1}})
&\leq& C(\alpha,\vec \beta, \vec \xi) K  \int_2^\infty dt \; t^{-1-\eps} \\
&\leq&
C(\alpha,\vec \beta, \gamma, \vec \xi, \eps) K.
\end{eqnarray*}
This concludes the proof.
\end{proof}
\noindent The next theorem (which first appeared in \cite{HK09}) combines the previous lemma with Theorem \ref{thm:c03} to provide the desired bound on the sum in (\ref{eq:c27}).
\begin{thm}\label{thm:c31}
Let $h \in \mathcal{M}(\alpha, \vec \beta,\gamma,\vec \xi, K)$, where $\vec \xi = (\xi_1, \ldots,$ $\xi_N) \in (\mt^N)_*$ and \linebreak $\vec \beta = ( \beta_1, \ldots, \beta_N) \in \mr_+^N$. Then for every $\eps, \tau > 0$ the following holds: If $\alpha >0$ then
\begin{equation}\label{eq:c37}\sum_{w \in \mathcal{Z}(h)}\frac{(1-|w|)^{\alpha+1+\tau}}{|w|^{(\gamma-\eps)_+}} \prod_{j=1}^N|w-\xi_j|^{(\beta_j-1+\tau)_+}\leq
C(\alpha,\vec \beta,\gamma,\vec \xi,\eps, \tau) K.\end{equation}
Furthermore, if $\alpha = 0$ then
\begin{equation}\label{eq:c38}\sum_{w \in \mathcal{Z}(h)}\frac{(1-|w|)}{|w|^{(\gamma-\eps)_+}} \prod_{j=1}^N|w-\xi_j|^{(\beta_j-1+\tau)_+}\leq
C(\vec \beta,\gamma,\vec \xi,\eps, \tau) K.\end{equation}
\end{thm}
\begin{proof}
Since the sum on the left-hand side of (\ref{eq:c37}) is bounded from above by
\begin{eqnarray*}
\sum_{w \in \mathcal{Z}(h), |w| \leq \frac 1 2 } \frac{1}{|w|^{(\gamma-\eps)_+}} + C(\gamma, \eps) \sum_{w \in \mathcal{Z}(h), |w| > \frac 1 2}
(1-|w|)^{\alpha+1+\tau} \prod_{j=1}^N|w-\xi_j|^{(\beta_j-1+\tau)_+},
\end{eqnarray*}
we see that the proof of (\ref{eq:c37}) is an immediate consequence of estimate (\ref{eq:c35}) and Lemma \ref{lem:c04}. The proof of (\ref{eq:c38}) is analogous starting from estimate (\ref{eq:c36}).
\end{proof}

\section{Eigenvalue estimates via the complex analysis approach}\label{sec:complex}

Applying the results obtained in the previous two chapters we derive estimates on the discrete spectrum of linear operators satisfying Assumption \ref{ass:1}. In particular, we present precise estimates on the discrete spectrum of perturbations of bounded and non-negative selfadjoint operators, respectively. Some of the material in this section is taken from \cite{Hans_diss}.

\subsection{Bounded operators - a general result}\label{sec:bounded-operators}

Throughout this section we make the following
\begin{assumption}
$Z_0$ and $Z$ are bounded operators in $\hil$, satisfying
\begin{itemize}
\item[(i)]  $M=Z-Z_0 \in \mathcal{S}_p(\hil)$ for some $p>0$.
\item[(ii)] $ \sigma_d(Z) = \sigma(Z) \cap \rho(Z_0)$.
\item[(iii)] $M_1$ and $M_2$ are two fixed bounded operators on $\hil$ such that $M= M_1M_2$ and  $M_2R_{Z_0}(a)M_1 \in \mathcal{S}_p(\hil)$ for every $a \in \hat{\rho}(Z_0)$.
\item[(iv)] $\hat{\rho}(Z_0)$ is conformally equivalent to the unit disk, that is there exists a (necessarily unique) mapping $\phi:\md\rightarrow \hat{\rho}(Z_0)$
with $\phi(0)=\infty$.
\end{itemize}
\end{assumption}

\begin{rem}
We note that, if assumption (i) holds, assumption (iii) will automatically hold if we take $M_1=I,M_2=M$. However, sometimes other factorizations
of $M$ will yield stronger results, and for an arbitrary factorization $M=M_1M_2$ it is not true that (i) implies (iii).
\end{rem}

As we have seen in Section \ref{sec:pert-determ}, the perturbation determinant
$$d=d_\infty^{Z,Z_0} : \hat{\rho}(Z_0) \to \mc, \qquad d_\infty^{Z,Z_0}(\lambda) = \detp(I-(Z-Z_0)R_{Z_0}(\lambda))$$
has the property that its zero set coincides with the discrete spectrum of $Z$, and $d(\infty)=1$.
We recall that, by (\ref{eq:comm}),
$$\detp(I-(Z-Z_0)R_{Z_0}(\lambda))=\detp(I-M_2R_{Z_0}(\lambda)M_1),$$
and estimate (\ref{eq:39}) showed that for $\lambda \in \hat{\rho}(Z_0)$ we have
\begin{equation}
  \label{eq:188}
  |d(\lambda)| \leq \exp \left( \Gamma_p \| M_2R_{Z_0}(\lambda)M_1 \|_{\mathcal{S}_p}^p \right),
\end{equation}
where the constant $\Gamma_p$ was introduced in (\ref{inequality}). Thus if we can show that, for suitable parameters
$K,\alpha,\xi_j,\beta_j$,
  \begin{equation}\label{eq:37a}
\|M_2R_{Z_0}(\phi(w))M_1\|_{\mathcal{S}_p}^p \leq \frac{K|w|^\gamma}{(1-|w|)^\alpha \prod_{j=1}^N |w-\xi_j|^{\beta_j}}, \quad w \in \md,
  \end{equation}
then we obtain
$$\log |(d\circ \phi)(w)|\leq \frac{\Gamma_pK|w|^\gamma}{(1-|w|)^\alpha \prod_{j=1}^N |w-\xi_j|^{\beta_j}}.$$
In other words, $d \circ \phi \in \mathcal{M}(\alpha,\vec \beta,\gamma, \vec \xi, \Gamma_p K)$. Therefore Theorem \ref{thm:c31} can be applied to $d\circ \phi$ and in this way we obtain the following result.

\begin{prop}\label{prop:13a}
Suppose (\ref{eq:37a}) holds,
where $\alpha, \beta_j, \gamma,K$ are non-negative and $\xi_j \in \mt$. Then for every $\eps, \tau > 0$ the following holds: If $\alpha>0$ then
  \begin{eqnarray}
     \sum_{\lambda \in \sigma_d(Z)} \frac{(1-|\phi^{-1}(\lambda)|)^{\alpha+1+\tau}}{|\phi^{-1}(\lambda)|^{(\gamma-\eps)_+}} \prod_{j=1}^N |\phi^{-1}(\lambda)-\xi_j|^{(\beta_j-1+\tau)_+} &\leq& C K, \qquad \quad \label{eq:38}
  \end{eqnarray}
where $C=C(\alpha, \vec \beta, \gamma, \vec \xi, \eps, \tau, p)$ and each eigenvalue is counted according to its multiplicity. Moreover, if $\alpha=0$ then the same inequality holds with $\alpha + 1 +\tau$ replaced by $1$.
\end{prop}
\begin{rem}
It remains an interesting open question whether (\ref{eq:38}) is still valid when $\tau=0$ and $\eps=0$, respectively. At the moment, even for the specific choices of  $Z_0$ considered below, we are neither able to answer the corresponding question in the affirmative nor to provide a suitable counterexample.
\end{rem}
\begin{convention}
In the remaining parts of this article, let us agree that whenever a sum involving eigenvalues is considered, each eigenvalue is counted according to its (algebraic) multiplicity.
\end{convention}
The previous result is very general but not very enlightening. To obtain useful information using Proposition \ref{prop:13a} we need to do two things:
\begin{itemize}
\item Obtain estimates of the form (\ref{eq:37a}) for the operator of interest.

\item Obtain estimates from below on the sum on the left-hand side of (\ref{eq:38}) in terms of simple functions of the eigenvalues, so as to obtain
interesting information on the eigenvalues.
\end{itemize}

Carrying out both of these steps requires us to impose restrictions on the spectrum of the unperturbed operator $Z_0$, thus enabling
us to express the mapping $\phi$ explicitly. In the next subsection we will concentrate on the case that $Z_0$ is self-adjoint, so that its spectrum is
real. There are, however, various other options for treating various classes of operators. We now demonstrate one of them.

\begin{example}\label{ex:normal}
Let $Z_0 \in \bdd(\hil)$ be normal. Assume that $\sigma(Z_0)=\sigma_{ess}(Z_0)= \overline{\md}$, and let $Z=Z_0+M$ where $M \in \mathcal{S}_p(\hil)$ (so with the notation  above we have $Z-Z_0=M_1M_2$, where $M_1=I$ and $M_2=M$). Note that $\sigma_d(Z)=\sigma(Z) \cap \overline{\md}^c$ by Proposition \ref{prop:2}.
A conformal map $\phi : \md \to \hat{\rho}(Z_0)$, mapping $0$ onto $\infty$, is given by $\phi(w)= w^{-1}$, and we have
$ M_2R_{Z_0}(w^{-1})M_1= MR_{Z_0}(w^{-1})$.
The spectral theorem for normal operators implies that
$$\|R_{Z_0}(w^{-1})\|=\dist(w^{-1},\mt)^{-1}=|w|(1-|w|)^{-1},$$
so we obtain
$$ \|MR_{Z_0}(\phi(w))\|_{\mathcal{S}_p}^p \leq \|M\|_{\mathcal{S}_p}^p |w|^p(1-|w|)^{-p}, \quad w \in \md.$$
Hence, applying Proposition \ref{prop:13a} with $\alpha=\gamma=p$, $\vec \beta = \vec 0$ and $K=\|M\|_{\mathcal{S}_p}^p$, we conclude that for $\tau \in (0,p)$ (choosing $\eps=\tau$)
$$\sum_{\lambda \in \sigma_d(Z)} \frac{(|\lambda|-1)^{p+1+\tau}  }{|\lambda|^{1+2\tau}} = \sum_{\lambda \in \sigma_d(Z)} \frac{(1-|\phi^{-1}(\lambda)|)^{p+1+\tau}  }{|\phi^{-1}(\lambda)|^{p-\tau}}\leq C(p,\tau) \|M\|_{\mathcal{S}_p}^p. $$
\end{example}

\begin{rem}
Actually, we will show below that the estimate in the previous example can be improved considerably using our alternative approach to eigenvalue estimates (see Example \ref{ex:normal2}).
\end{rem}

\subsection{Perturbations of bounded selfadjoint operators}\label{sec:bounded}

\noindent Throughout this section we assume that $A_0 \in \bdd(\hil)$ is selfadjoint with $\sigma(A_0)=[a,b]$,\footnote{In this section we are changing notation from $Z_0$ to $A_0$ (and from $Z$ to $A$), the reason being the specific choice we make for the spectrum of $A_0$. A similar remark will apply in Section \ref{sec:pert-nonn-oper}.} where $a<b$,  and that  $M=M_1M_2 \in \mathcal{S}_p(\hil)$ for some $p > 0,$ where $M_1$ and $M_2$ are bounded operators on $\hil$ satisfying
\begin{equation}
  \label{eq:211}
 M_2R_{A_0}(\lambda)M_1 \in \mathcal{S}_p(\hil), \quad \lambda \in \hat{\rho}(A_0).
\end{equation}
\noindent In particular, $A_0$ and $A=A_0+M$ satisfy Assumption \ref{ass:1}  by Remark \ref{rem:6} (with $Z_0=A_0$ and $Z=A$, respectively), and we have $$\sigma(A)=[a,b] \dotcup \sigma_d(A).$$
Let us define a conformal map $\phi_1 :\md \to \hat{\mc} \setminus [a,b]$, mapping $0$ onto $\infty$, by setting
\begin{equation}\label{eq:41}
 \quad \phi_1(w)= \frac{b-a}{4}(w + w^{-1}+2)+a, \quad w \in \md.
\end{equation}
To adapt Proposition \ref{prop:13a} to the present context we will need the following elementary but crucial inequalities, see Lemma 7 in \cite{HK09}.

\begin{lemma}\label{lem:c05}
For $w \in \md$  let $\phi_1(w)$ be defined by (\ref{eq:41}). Then
   \begin{equation*}
  \frac{b-a}{8} \frac{|w^2-1|(1-|w|)}{|w|} \leq \dist(\phi_1(w),[a,b]) \leq \frac{(b-a)(1+\sqrt{2})}{8} \frac{|w^2-1|(1-|w|)}{|w|}.
  \end{equation*}
\end{lemma}

\noindent In the following, we derive estimates on $\sigma_d(A)$ given the assumption that for every $\lambda \in \mc \setminus [a,b]$ we have
\begin{equation}
    \label{eq:42}
    \| M_2R_{A_0}(\lambda)M_1\|_{\mathcal{S}_p}^p \leq K \frac{|\lambda-a|^{\beta}|\lambda-b|^{\beta}}{\dist(\lambda,[a,b])^\alpha},
\end{equation}
where $\alpha, K \in \mr_+$, $\beta \in \mr$ and $\alpha > 2\beta$. Of course, one could imagine different assumptions on the norm of $M_2R_{A_0}(\lambda)M_1$, e.g., a different behavior at the boundary points $a$ and $b$, but the choice above is sufficiently general for the applications we have in mind.

\begin{thm}\label{thm:3}
  With the assumptions and notations from above, suppose that \linebreak  $M_2R_{A_0}(\lambda)M_1$ satisfies estimate (\ref{eq:42}) for every $\lambda \in \mc \setminus [a,b]$.  Let $\tau \in (0,1)$ and define
\begin{equation}
  \label{eq:43}
  \begin{array}{rcl}
    \eta_1 & = & \alpha + 1 +\tau, \\
    \eta_2 & = & (\alpha-2\beta-1+\tau)_+.
  \end{array}
\end{equation}
Then the following holds: If $\alpha >0$ then
\begin{equation}
  \label{eq:44}
  \sum_{\lambda \in \sigma_d(A)} \frac{\dist(\lambda,[a,b])^{\eta_1}}{(|b-\lambda||a-\lambda|)^{\frac{\eta_1-\eta_2}{2}}} \leq C(\alpha,\beta,\tau,p) (b-a)^{\eta_2-\alpha+2\beta} K.
\end{equation}
Moreover, if $\alpha=0$ then the same inequality holds with $\eta_1$ replaced by $1$.
\end{thm}
\begin{proof}
We consider the case $\alpha>0$ only. As above, let
$$\lambda=\phi_1(w)=\frac{b-a}{4}(w + w^{-1}+2)+a, \qquad w \in \md.$$
Then a short computation shows that
  \begin{equation}
    \label{eq:45}
    |a-\lambda|= \frac{b-a}{4} \frac{|w+1|^2}{|w|} \quad \text{and} \quad
    |b-\lambda|= \frac{b-a}{4} \frac{|w-1|^2}{|w|}.
  \end{equation}
Using the last two identities and Lemma \ref{lem:c05}, the assumption in (\ref{eq:42}) can be rewritten as
\begin{equation}
  \label{eq:46}
  \|M_2R_{A_0}(\lambda)M_1\|_{\mathcal{S}_p}^p  \leq \frac{C(\alpha,\beta) K}{(b-a)^{\alpha-2\beta}}  \frac{|w|^{\alpha-2\beta}}{(1-|w|)^\alpha|w^2-1|^{\alpha-2\beta}}.
\end{equation}
Let $\eps,\tau > 0$ and let $\eta_1, \eta_2$ be defined by (\ref{eq:43}). Then Proposition \ref{prop:13a} implies that
\begin{equation}
 \sum_{\lambda \in \sigma_d(A)} \frac{(1-|\phi_1^{-1}(\lambda)|)^{\eta_1}}{|\phi_1^{-1}(\lambda)|^{(\alpha-2\beta-\eps)_+}} |(\phi_1^{-1}(\lambda))^2-1|^{\eta_2}
\leq \frac{C(\alpha,\beta,\eps,\tau,p) K}{(b-a)^{\alpha-2\beta}}.
\end{equation}
Restricting $\tau$ to the interval $(0,1)$ and setting $\eps=1-\tau$, the last inequality can be rewritten as
\begin{equation}
  \label{eq:47}
 \sum_{\lambda \in \sigma_d(A)} \frac{(1-|\phi_1^{-1}(\lambda)|)^{\eta_1}}{|\phi_1^{-1}(\lambda)|^{\eta_2}} |(\phi_1^{-1}(\lambda))^2-1|^{\eta_2}
\leq \frac{C(\alpha,\beta,\tau,p) K}{(b-a)^{\alpha-2\beta}}.
\end{equation}
By (\ref{eq:45}) we have
\begin{equation}
  \label{eq:48}
  |(\phi_1^{-1}(\lambda))^2-1| = \frac{4}{b-a} |\phi_1^{-1}(\lambda)| (|\lambda-a||\lambda-b|)^{{1}/{2}},
\end{equation}
and by Lemma \ref{lem:c05}, we obtain
\begin{eqnarray}
  (1-|\phi_1^{-1}(\lambda)|) &\geq& \frac{8}{(1+\sqrt{2})(b-a)} \frac{|\phi_1^{-1}(\lambda)|\dist(\lambda,[a,b])}{|(\phi_1^{-1}(\lambda))^2-1|} \nonumber \\
&=&  \frac{2}{(1+\sqrt{2})} \frac{\dist(\lambda,[a,b])}{(|\lambda-a||\lambda-b|)^{1/2}}   \label{eq:49}.
\end{eqnarray}
Inserting (\ref{eq:49}) and (\ref{eq:48}) into (\ref{eq:47}) concludes the proof.
\end{proof}

\begin{rem}
The left- and right-hand sides of (\ref{eq:49}) are actually equivalent (meaning that the same inequality, with another constant, holds in the other direction as well), so no essential information gets lost in this estimate.
\end{rem}

\begin{rem}\label{example1}
A nice way to illustrate the consequences of the finiteness of the sum in (\ref{eq:44}) is to consider sequences $\{ \lambda_k\}$ of isolated eigenvalues of $A$ converging to some $\lambda^* \in [a,b]$. Taking a subsequence, we can suppose that one of the following options holds:
\begin{equation*}
\begin{array}{clcl}
\mbox{(i.a)} & \lambda^* = a \text{ and } \Re(\lambda_k)\leq a. \quad &
\mbox{(i.b)} & \lambda^* = b \text{ and } \Re(\lambda_k)\geq b. \\[4pt]
\mbox{(ii.a)} & \lambda^* = a \text{ and }\Re(\lambda_k)> a. \quad &
\mbox{(ii.b)} & \lambda^* = b \text{ and } \Re(\lambda_k)< b.\\[4pt]
\mbox{(iii)} & \lambda^* \in (a,b). & &
\end{array}
\end{equation*}
It is sufficient to consider the cases (i.a), (ii.a) and (iii) only. In case (i.a), since $\dist(\lambda_k,[a,b])=|\lambda_k-a|$, (\ref{eq:44}) implies the finiteness of $\sum_k |\lambda_k-a|^{(\eta_1+\eta_2)/2}$ showing that any such sequence must converge to $a$ sufficiently fast. Similarly, in case (ii.a), (\ref{eq:44}) implies the finiteness of $\sum_k \frac{|\Im(\lambda_k)|^{\eta_1}}{|\lambda_k-a|^{(\eta_1-\eta_2)/2}}$. Finally, in case (iii), we obtain the finiteness of $\sum_k |\Im(\lambda_k)|^{\eta_1}$, showing that the sequence must converge to the real line sufficiently fast.
\end{rem}

\noindent Theorem \ref{thm:3} still relies on a quantitative estimate on the $\mathcal{S}_p$-norm of the operator $M_2R_{A_0}(\lambda)M_1$.
In particular applications, one wants to choose the decomposition $M=M_1M_2$ so as to obtain an estimate on $M_2R_{A_0}(\lambda)M_1$ as strong as
possible (we will indicate this process when considering Jacobi operators in Chapter \ref{sec:jacobi}).
Let us note, however, that we can always take the `trivial' decomposition $M_1=I$ and $M_2=M$, and use the bound
$$ \|MR_{A_0}(\lambda)\|_{\mathcal{S}_p}^p \leq \|M\|_{\mathcal{S}_p}^p \|R_{A_0}(\lambda)\|^p \leq \frac{\|M\|_{\mathcal{S}_p}^p}{\dist(\lambda,[a,b])^p},$$
so that we obtain the following estimates.
\begin{cor}\label{cor:2}
Let $A_0 \in \bdd(\hil)$ be selfadjoint with $\sigma(A_0)=[a,b]$ and let $A=A_0+M$ where $M \in \mathcal{S}_p(\hil)$. Then for $\tau \in (0,1)$ the following holds: If $p \geq 1 - \tau$ then
\begin{equation}
\label{eq:50}
  \sum_{\lambda \in \sigma_d(A)} \frac{\dist(\lambda,[a,b])^{p+1+\tau}}{|b-\lambda||a-\lambda|} \leq C(p,\tau) (b-a)^{-1+\tau} \|M\|_{\mathcal{S}_p}^p.
\end{equation}
Moreover, if $0<p<1 - \tau$ then
\begin{equation}
\label{eq:51}
  \sum_{\lambda \in \sigma_d(A)} \left( \frac{\dist(\lambda,[a,b])}{|b-\lambda|^{1/2}|a-\lambda|^{1/2}}\right)^{p+1+\tau} \leq C(p,\tau) (b-a)^{-p} \|M\|_{\mathcal{S}_p}^p.
\end{equation}
\end{cor}
\begin{proof}
  Apply Theorem \ref{thm:3} with $M_1=I$, $M_2=M$, $K= \|M\|_{\mathcal{S}_p}^p$, $\alpha =p$ and $\beta = 0$.
\end{proof}

\begin{rem}\label{rem:8}
In view of estimate (\ref{eq:51}), we should mention that in general it is not possible to infer the finiteness of the sum
\begin{equation}
  \label{eq:193}
    \sum_{\lambda \in \sigma_d(A_0+M)} \left( \frac{\dist(\lambda,[a,b])}{|b-\lambda|^{1/2}|a-\lambda|^{1/2}}\right)^\gamma, \quad \text{ where } \gamma < 1,
\end{equation}
from the mere assumption that $M \in \mathcal{S}_p(\hil)$ for some $p>0$. Indeed, if $A_0$ is the free Jacobi operator, then for every $\gamma<1$ we can construct a rank one perturbation $M$ such that the sum in (\ref{eq:193}) diverges, see Appendix C in \cite{Hans_diss}.
\end{rem}
\begin{rem}\label{rem:borichev}
We note that a slightly weaker version of the previous theorem has first been obtained by Borichev, Golinskii and Kupin \cite{Borichev08} in the context of Jacobi operators. They used Theorem \ref{thm:c05} instead of Theorem \ref{thm:c31} in its derivation, which resulted in a constant on the right-hand side depending on the operator $A$ in some unspecified way.
\end{rem}

We will return to Corollary \ref{cor:2} in Chapter \ref{sec:comparison}, where we will compare it with some related results obtained via our alternative approach to eigenvalues estimates, which will be described in Chapter \ref{sec:operator}.

\subsection{Unbounded operators - a general result}\label{sec:an-abstract-estimate}

\noindent We are now interested in applying similar considerations to the study of eigenvalues of unbounded operators. Throughout this section we make the following
\begin{assumption}
$Z_0$ and $Z$ are operators in $\hil$  satisfying
\begin{itemize}
\item[(i)] $Z,Z_0 \in \cld(\hil)$ are densely defined with $\rho(Z_0) \cap \rho(Z) \neq \emptyset$.
\item[(ii)]  $R_Z(b)-R_{Z_0}(b) \in \mathcal{S}_p(\hil)$
for some $b \in \rho(Z_0) \cap \rho(Z)$ and some $p >0$.
\item[(iii)] $ \sigma_d(Z) = \sigma(Z) \cap \rho(Z_0)$.
\item[(iv)] $\rho(Z_0)$ is conformally equivalent to the unit disk. More precisely, there exists a conformal mapping $\psi : \md \to  \rho(Z_0)$ with $\psi(0)=a$, where $a$ is some fixed element of $\rho(Z_0) \cap \rho(Z)$.
\end{itemize}
\end{assumption}

The analysis of the discrete spectrum of $Z$ is quite similar to the analysis made in Section \ref{sec:bounded-operators}, with the only difference that the discrete spectrum of $Z$ now coincides with the
zero set of the perturbation determinant
$$d_a^{Z,Z_0} : \rho(Z_0) \to \mc, \quad d_a^{Z,Z_0}(\lambda)=\detp(I-[R_Z(a)-R_{Z_0}(a)][(a-\lambda)^{-1}-R_{Z_0}(a)]^{-1}),$$
compare Section \ref{sec:pert-determ}. In particular, we can use the same line of reasoning as in Section \ref{sec:bounded-operators} to obtain the following result.

\begin{prop}\label{prop:unbounded}
Suppose that for some non-negative constants $K,\alpha,\beta_j, \gamma$ and some $\xi_j \in \mt$ we have for every $w \in \md$
\begin{small}
  \begin{equation}
 \|[R_Z(a)-R_{Z_0}(a)][(a-\psi(w))^{-1}-R_{Z_0}(a)]^{-1}\|_{\mathcal{S}_p}^p
\leq  \frac{K|w|^\gamma}{(1-|w|)^\alpha \prod_{j=1}^N |w-\xi_j|^{\beta_j}}. \label{eq:prop:unb}
  \end{equation}
\end{small}
Then for every $\eps, \tau > 0$ the following holds: If $\alpha>0$ then
  \begin{eqnarray*}
     \sum_{\lambda \in \sigma_d(Z)} \frac{(1-|\psi^{-1}(\lambda)|)^{\alpha+1+\tau}}{|\psi^{-1}(\lambda)|^{(\gamma-\eps)_+}} \prod_{j=1}^N |\psi^{-1}(\lambda)-\xi_j|^{(\beta_j-1+\tau)_+} &\leq& C(\alpha, \vec \beta, \gamma, \vec \xi, \eps, \tau, p) K. 
  \end{eqnarray*}
Moreover, if $\alpha=0$ then the same inequality holds with $\alpha + 1 +\tau$ replaced by $1$.
\end{prop}

\begin{rem}\label{rem:unbound}
  If $Z = Z_0 + M$ where $M$ is $Z_0$-compact, then we can use the second resolvent identity (\ref{eq:10}) to obtain
$$ [R_Z(a)-R_{Z_0}(a)][(a-\psi(w))^{-1}-R_{Z_0}(a)]^{-1} = (a-\psi(w))R_Z(a)MR_{Z_0}(\psi(w)),$$
so in order to satisfy the conditions of the last proposition we need a good control of the $\mathcal{S}_p$-norm of $MR_{Z_0}(\psi(w))$. We will return to this topic in the next section.
\end{rem}



\noindent We can obtain a more ``explicit'' version of Proposition \ref{prop:unbounded}  using Koebe's distortion theorem, see \cite{b_Pommerenke92}, page 9.
\begin{thm}\label{thm:c01}
 Let $\vphi : \md \to \vphi(\md)$ be conformal. Then
 \begin{equation}\label{eq:c05}
   \frac 1 4 |\vphi'(w)|(1-|w|) \leq \dist( \vphi(w), \partial \vphi(\md)) \leq 2 |\vphi'(w)|(1-|w|)
 \end{equation}
for $w \in \md$.
\end{thm}

\begin{cor}\label{cor:1}
Suppose that (\ref{eq:prop:unb}) is satisfied for some non-negative constants $K,\alpha,\beta_j, \gamma$ and some $\xi_j \in \mt$.
Then for every $\eps, \tau > 0$ the following holds: If $\alpha>0$ then
  \begin{equation}
\sum_{\lambda \in \sigma_d(Z)} \frac{(\dist(\lambda,\partial \sigma(Z_0)) |(\psi^{-1})'(\lambda)|)^{\alpha+1+\tau}}{|\psi^{-1}(\lambda)|^{(\gamma-\eps)_+}} \prod_{j=1}^N |\psi^{-1}(\lambda)-\xi_j|^{(\beta_j-1+\tau)_+} \leq C K, \label{eq:183}
  \end{equation}
where $C=C(\alpha, \vec \beta, \gamma, \vec \xi, \eps, \tau,p)$. Moreover, if $\alpha=0$ then the same inequality holds with $\alpha + 1 +\tau$ replaced by $1$.
\end{cor}

\begin{proof}
Use Proposition \ref{prop:unbounded} and the fact that, by Koebe's distortion theorem, for $\lambda \in \rho(Z_0)=\psi(\md)$ we have
\begin{equation*}
4 \dist(\lambda,\partial \rho(Z_0)) \geq  (1-|\psi^{-1}(\lambda)|)|\psi'(\psi^{-1}(\lambda))|\geq \frac 1 2 \dist(\lambda,\partial \rho(Z_0)).
\end{equation*}
Now note that $\partial \rho(Z_0)= \partial \sigma(Z_0)$.
\end{proof}

\subsection{Perturbations of non-negative operators}\label{sec:pert-nonn-oper}

\noindent In this section we assume that $H_0$ is a selfadjoint operator in $\hil$ with $\sigma(H_0)=[0,\infty)$, and $H \in \cld(\hil)$ is densely defined with
\begin{equation}
  \label{eq:52}
  R_{H}(u)-R_{H_0}(u) \in \mathcal{S}_p(\hil)
\end{equation}
for some  $u \in \rho(H_0) \cap \rho(H)$ (which we assume to be non-empty) and some fixed $p \in (0,\infty)$. In particular, by Remark \ref{rem:6}, $H_0$ and $H$ satisfy Assumption \ref{ass:1}  (with $Z_0=H_0$ and $Z=H$, respectively) and we have $$\sigma(H)=[0,\infty) \dotcup \sigma_d(H).$$

\begin{rem}
Given the above assumptions we could use Corollary \ref{cor:1} to derive a quite explicit estimate on the discrete eigenvalues of $H$ in terms of the $\mathcal{S}_p$-norm of $R_H(u)-R_{H_0}(u)$, see \cite{Hans_diss} Theorem 3.3.1. However, we decided against presenting this estimate in this review since it is weaker than an analogous estimate we can obtain using our alternative approach to eigenvalue estimates (see Theorem \ref{thm:num_non}).
\end{rem}

Actually, in the following we will restrict ourselves to a less general but much simpler situation: We will assume that $H=H_0+M$ where $M$ is $H_0$-compact and  $MR_{H_0}(u) \in \mathcal{S}_p(\hil)$ for some (and hence all) $u \in \rho(H_0)$. Moreover, we will assume that there exists $\omega \leq 0$ such that
\begin{equation}\label{eq:omega1}
  \{ \lambda : \Re(\lambda) < \omega \} \subset \rho(H)
\end{equation}
and that there exists $C_0(\omega) > 0$ such that for every $\lambda$ with $\Re(\lambda) < \omega$ we have
\begin{equation}\label{eq:omega2}
  \|R_H(\lambda)\| \leq \frac{C_0(\omega)}{|\Re(\lambda)-\omega|}.
\end{equation}
\begin{rem}
  The existence of some $\omega$ with the above properties is actually implied by the $H_0$-compactness of $M$ (see, e.g., the discussion in \cite{Hans_diss} Section 3.3).
  If the operator $H$ is m-sectorial with vertex $\gamma \leq 0$, then we can choose $\omega=\gamma$ and $C_0(\omega)=1$.  For instance, the  Schr\"odinger operators considered in Chapter \ref{sec:schroedinger} will be $m$-sectorial.
\end{rem}

Now let us fix some $a < \omega$ and choose $b>0$ such that $a=-b^2$. For later purposes let us note that a conformal mapping $\psi_1$ of $\md$ onto $\mc \setminus [0,\infty)$, which maps $0$ onto $a$, is given by
\begin{equation}\label{eq:53}
\psi_1(w)= a\left(\frac{1+w}{1-w}\right)^2, \quad \psi_1^{-1}(\lambda)=\frac{\sqrt{-\lambda}-b}{\sqrt{-\lambda}+b}.
\end{equation}
Here the square root is chosen such that $\Re(\sqrt{-\lambda})>0$ for $\lambda \in \mc \setminus [0,\infty)$. In particular, we note that $\psi_1(-1)=0$ and $\psi_1(1)=\infty$.

In the following,  we will derive a first estimate on $\sigma_d(H)$ (which will not use estimate (\ref{eq:omega2})) given the quantitative assumption that for every $\lambda \in \mc \setminus [0,\infty)$ we have
\begin{equation}
  \label{eq:70}
  \|R_H(a)MR_{H_0}(\lambda)\|_{\mathcal{S}_p}^p \leq \frac{K |\lambda|^\beta}{\dist(\lambda,[0,\infty))^\alpha},
\end{equation}
where $\alpha,K$ are non-negative and $\beta \in \mr$ (note that the values of the constants might also depend on the choice of $a$).
\begin{thm}\label{thm:5}
  With the assumptions and notation from above, assume that the operator $R_H(a)MR_{H_0}(\lambda)$ satisfies assumption (\ref{eq:70}). Let $\eps, \tau > 0$ and define
\begin{equation}
  \label{eq:71}
  \begin{array}{rcl}
    \eta_1&=&\alpha+1+\tau, \\
    \eta_2&=& ((\alpha-2\beta)_+-1+\tau)_+, \\
    \eta_3&=&((2p-3\alpha+2\beta)_+-1+\tau)_+, \\
    \eta_4&=&(p-\eps)_+.
  \end{array}
\end{equation}
Then the following holds: If $\alpha > 0$ then
  \begin{eqnarray}
     \sum_{\lambda \in \sigma_d(H)} \frac{\dist(\lambda,[0,\infty))^{\eta_1}}{|\lambda|^{\frac{\eta_1-\eta_2}{2}}(|\lambda|+|a|)^{\eta_1-\eta_4+\frac{\eta_2+\eta_3}{2}}|\lambda-a|^{\eta_4}} \leq C |a|^{-(\frac{\eta_1+\eta_3}{2}-p+\alpha-\beta)} K, \label{eq:72}
  \end{eqnarray}
where $C=C(\alpha,\beta, p, \eps, \tau)$. Furthermore, if $\alpha=0$ then the same inequality holds with $\eta_1$ replaced by $1$.
\end{thm}
\begin{rem}
 The parameter $\frac{\eta_1+\eta_3}{2}-p+\alpha-\beta$ is positive, as a short computation shows.
\end{rem}
\begin{proof}[Proof of Theorem \ref{thm:5}]
We consider the case $\alpha > 0$ only. Let $\lambda=\psi_1(w)=a(\frac{1+w}{1-w})^2$ and note that
  \begin{equation*}
    \psi_1(w)-a= \frac{4aw}{(1-w)^2}.
  \end{equation*}
Together with assumption (\ref{eq:70}), the last identity implies that
\begin{equation}\label{eq:73}
  |\psi_1(w)-a|^p \|R_H(a)MR_{H_0}(\psi_1(w))\|_{\mathcal{S}_p}^p \leq \frac{4^p|a|^p|w|^p}{|1-w|^{2p}}\frac{K |\psi_1(w)|^{\beta}}{\dist(\psi_1(w),[0,\infty))^\alpha}.
\end{equation}
Since $\psi_1'(w)= \frac{4a(1+w)}{(1-w)^3}$, we obtain from Theorem \ref{thm:c01} that
\begin{equation*}
  \dist(\psi_1(w),[0,\infty)) \geq |a| \frac{|1+w| (1-|w|)}{|1-w|^3}.
\end{equation*}
Using this inequality and the definition of $\psi_1$ we see that the right-hand side of (\ref{eq:73}) is bounded from above by
\begin{equation*}
\frac{4^pK|a|^{p-\alpha+\beta}|w|^p}{(1-|w|)^\alpha |1+w|^{\alpha-2\beta} |1-w|^{2p-3\alpha+2\beta} }.
\end{equation*}
Applying Corollary \ref{cor:1}, taking Remark \ref{rem:unbound} into account,  we thus obtain that for $\eps,\tau>0$,
\begin{equation}\label{eq:35}
   \sum_{\lambda \in \sigma_d(H)} \frac{|\dist(\lambda,[0,\infty))(\psi_1^{-1})'(\lambda)|^{\eta_1}}{|\psi_1^{-1}(\lambda)|^{\eta_4}} |\psi_1^{-1}(\lambda)+1|^{\eta_2} |\psi_1^{-1}(\lambda)-1|^{\eta_3} \leq C |a|^{p-\alpha+\beta} K,
\end{equation}
where $C=C(\alpha,\beta,p,\eps, \tau)$. Recall that $\psi_1^{-1}(\lambda)= \frac{\sqrt{-\lambda}-b}{\sqrt{-\lambda}+b}$ where $b = \sqrt{-a}$. Since
$$ (\psi_1^{-1})'(\lambda)=\frac{-b}{\sqrt{-\lambda}(\sqrt{-\lambda}+b)^2}$$
and
$$ \psi_1^{-1}(\lambda)-1=\frac{-2b}{\sqrt{-\lambda}+b}, \quad \psi_1^{-1}(\lambda)+1=\frac{2\sqrt{-\lambda}}{\sqrt{-\lambda}+b},$$
estimate (\ref{eq:35}) implies that
\begin{eqnarray*}
  \sum_{\lambda \in \sigma_d(H)} \frac{\dist(\lambda,[0,\infty))^{\eta_1}}{|\lambda|^{\frac{\eta_1-\eta_2}{2}} |\sqrt{-\lambda}+b|^{2\eta_1+\eta_2+\eta_3-\eta_4}|\sqrt{-\lambda}-b|^{\eta_4}}
\leq C |a|^{p-\alpha+\beta-\frac{\eta_1+\eta_3}{2}} K.
\end{eqnarray*}
We conclude the proof by noting that
\begin{equation*}
   |\sqrt{-\lambda}-b|= \frac{|\lambda-a|}{|\sqrt{-\lambda}+b|}
\end{equation*}
and
\begin{equation*}
  |\sqrt{-\lambda}+b| \leq (|\lambda|^{1/2}+b) \leq 2 (|\lambda|+|a|)^{1/2}.
\end{equation*}
\end{proof}

\begin{rem}
Analogous to our discussion in Remark \ref{example1}, let us consider the consequences of estimate (\ref{eq:72}) on the discrete spectrum of $H$ in a little more detail. To this end, let  $\{ \lambda_k\}$ be a sequence of isolated eigenvalues of $H$ converging to some $\lambda^* \in [0,\infty)$. Taking a subsequence, we can suppose that one of the following options holds:
\begin{equation*}
\begin{array}{llllll}
\mbox{(i)} & \lambda^* = 0 \text{ and } \Re(\lambda_k)\leq 0 \quad &
\mbox{(ii)} & \lambda^* = 0 \text{ and } \Re(\lambda_k)> 0 \quad &
\mbox{(iii)} & \lambda^*>0.
\end{array}
\end{equation*}
In case (i), since $\dist(\lambda_k,[0,\infty))=|\lambda_k|$, (\ref{eq:72}) implies the finiteness of $$\sum_k |\lambda_k|^{(\eta_1+\eta_2)/2},$$
so any such sequence must converge to $0$ sufficiently fast. Similarly, in case (ii), (\ref{eq:72}) implies the finiteness of $\sum_k {|\Im(\lambda_k)|^{\eta_1}}{|\lambda_k|^{-(\eta_1-\eta_2)/2}}$, and in case (iii) we obtain the finiteness of $\sum_k |\Im(\lambda_k)|^{\eta_1}$, which shows that any such sequence must converge to the real line sufficiently fast. Estimate (\ref{eq:72}) also provides information about divergent sequences of eigenvalues. For example, if $\{ \lambda_k\}$ is an infinite sequence of eigenvalues which stays bounded away from $[0,\infty)$, that is, $\dist(\lambda_k,[0,\infty)) \geq \delta$ for some $\delta > 0$ and all $k$, then (\ref{eq:72}) implies that
$$ \sum_k \frac 1 {|\lambda_k|^{(3\eta_1+\eta_3)/2}} < \infty,$$
which shows that the sequence $\{ \lambda_k\}$ must diverge to infinity sufficiently fast.
\end{rem}

 Estimate (\ref{eq:72}) provides us with a family of inequalities parameterized by $a < \omega$. By considering an average of all these inequalities, i.e., by multiplying both sides of  (\ref{eq:72}) with an $a$-dependent weight and integrating with respect to $a$, it is possible to extract some more information on $\sigma_d(H)$. Of course, in this context, we have to be aware that the constants and parameters on the right-hand side of (\ref{eq:72}) may still depend on $a$. We can use the estimate (\ref{eq:omega2}) to get rid of this dependence.
\begin{thm}\label{cor:3}
Let $\omega \leq 0$ and $C_0=C_0(\omega)>0$ be chosen as in (\ref{eq:omega1}) and (\ref{eq:omega2}), respectively, and assume that for all $\lambda \in \mc \setminus [0,\infty)$ we have
  \begin{equation}\label{eq:77}
\|MR_{H_0}(\lambda)\|_{\mathcal{S}_p}^p \leq \frac{K|\lambda|^{\beta}}{\dist(\lambda,[0,\infty))^\alpha},
  \end{equation}
where $K \geq 0$, $\alpha > 0$ and $\beta \in \mr$. Let $\tau > 0$ and define
\begin{equation}
  \begin{array}{rcl}
    \eta_0 &=& -\alpha + \beta+\tau, \\
    \eta_1&=&\alpha+1+\tau, \\
    \eta_2&=&((\alpha-2\beta)_+-1+\tau)_+.
     \end{array}
\end{equation}
Then the following holds: If $\omega < 0$ then
  \begin{equation}
  \sum_{\lambda \in \sigma_d(H)} \frac{\dist(\lambda,[0,\infty))^{\eta_1}}{|\lambda|^{\frac{\eta_1-\eta_2}{2}}(|\lambda|+|\omega|)^{\eta_0+\frac{\eta_1+\eta_2}{2}}}
 \leq C \frac{ C_0^p  K}{ |\omega|^{\tau}}. \qquad \label{eq:78}
  \end{equation}
If $\omega = 0$ then for $s> 0$
  \begin{equation}
  \sum_{\lambda \in \sigma_d(H), |\lambda| > s} \frac{\dist(\lambda,[0,\infty))^{\eta_1}}{|\lambda|^{\beta+1+2\tau}}
 + \sum_{\lambda \in \sigma_d(H), |\lambda| \leq s} \frac{\dist(\lambda,[0,\infty))^{\eta_1}}{|\lambda|^{\beta+1}s^{2\tau}}
\leq C \frac{C_0^p K}{  s^{\tau}} . \qquad \label{eq:79}
  \end{equation}
In both cases, $C=C(\alpha,\beta, p,\tau)$. Moreover, if $\alpha=0$ then in (\ref{eq:78}) and (\ref{eq:79})  we can replace $\eta_1$ by $1$.
\end{thm}

\begin{proof}
 By (\ref{eq:omega2}) we have, for $a<\omega$,
$ \| R_H(a)\| \leq {C_0} {|a-\omega|^{-1}}$, so (\ref{eq:77}) implies that
\begin{eqnarray}
  \label{eq:80}
\|R_H(a)MR_{H_0}(\lambda)\|_{\mathcal{S}_p}^p &\leq& \frac {C_0^pK} {|a-\omega|^p} \frac{ |\lambda|^{\beta}}{\dist(\lambda,[0,\infty))^\alpha}.
\end{eqnarray}
For $\eps, \tau > 0$, let $\eta_j$, where $j =1,\ldots,4$, be defined by (\ref{eq:71}). Then  Theorem \ref{thm:5} implies that
  \begin{eqnarray*}
\sum_{\lambda \in \sigma_d(H)} \frac{\dist(\lambda,[0,\infty))^{\eta_1}}{|\lambda|^{\frac{\eta_1-\eta_2}{2}}(|\lambda|+|a|)^{\eta_1-\eta_4+\frac{\eta_2+\eta_3}{2}}|\lambda-a|^{\eta_4}}
\leq  \frac{C_0^p C(\alpha,\beta,p,\eps,\tau)K}{|a|^{\frac{\eta_1+\eta_3}{2}-p+\alpha-\beta}|a-\omega|^p}.
\end{eqnarray*}
Setting $\eps=\tau$ and using that $|\lambda-a|\leq (|\lambda|+|a|)$ the last inequality implies that
  \begin{eqnarray}
\sum_{\lambda \in \sigma_d(H)} \frac{\dist(\lambda,[0,\infty))^{\eta_1}}{|\lambda|^{\frac{\eta_1-\eta_2}{2}}(|\lambda|+|a|)^{\eta_1+\frac{\eta_2+\eta_3}{2}}}
\leq \frac{C_0^p  C(\alpha,\beta,p,\tau)K}{|a|^{\frac{\eta_1+\eta_3}{2}-p+\alpha-\beta}|a-\omega|^p}. \label{eq:81}
\end{eqnarray}
To simplify notation, we set $r=|a| (>|\omega|)$, $C= C(\alpha,\beta,p,\tau)$,
\begin{eqnarray*}
  \varphi_1 = \frac{\eta_1+\eta_3}{2}-p+\alpha-\beta \quad \text{and} \quad
  \varphi_2 = \eta_1 + \frac{\eta_2+ \eta_3}{2}.
\end{eqnarray*}
Note that $\varphi_1, \varphi_2  > 0$. Now let us introduce some constant $s=s(\omega)$. More precisely, we choose $s=0$ if $|\omega| > 0$ and $s>0$ if $\omega = 0$.
Then we can rewrite (\ref{eq:81}) as follows
\begin{equation}\label{eq:82}
\sum_{\lambda \in \sigma_d(H)} \frac{\dist(\lambda,[0,\infty))^{\eta_1}r^{\varphi_1-1+\tau} (r-|\omega|)^{p}}{|\lambda|^{\frac{\eta_1-\eta_2}{2}} (|\lambda|+r)^{\varphi_2}(s+r)^{2\tau}}
\leq \frac{C_0^p  C K}{r^{1-\tau}(s+r)^{2\tau}} .
\end{equation}
Next, we integrate both sides of the last inequality with respect to $r \in (|\omega|,\infty)$. We obtain for the right-hand side 
\begin{equation}
  \label{eq:83}
  \int_{|\omega|}^\infty  \frac{dr}{r^{1-\tau}(s+r)^{2\tau}} = \left\{
    \begin{array}{cl}
      \frac 1 {\tau |\omega|^\tau}, & |\omega| > 0 \text{ and } s = 0 \\[4pt]
      \frac{C(\tau)}{s^\tau}, & \omega = 0 \text{ and } s > 0.
    \end{array}\right.
\end{equation}
Integrating the left-hand side of (\ref{eq:82}), interchanging sum and integral, it follows that
\begin{eqnarray}
&&  \int_{|\omega|}^\infty dr \left(\sum_{\lambda \in \sigma_d(H)} \frac{\dist(\lambda,[0,\infty))^{\eta_1}r^{\varphi_1-1+\tau} (r-|\omega|)^{p}}{|\lambda|^{\frac{\eta_1-\eta_2}{2}} (|\lambda|+r)^{\varphi_2}(s+r)^{2\tau}}  \right) \nonumber \\
&=& \sum_{\lambda \in \sigma_d(H)} \frac{\dist(\lambda,[0,\infty))^{\eta_1}}{|\lambda|^{\frac{\eta_1-\eta_2}{2}}} \int_{|\omega|}^\infty dr \frac{(r-|\omega|)^{p}r^{\varphi_1-1+\tau}}{(|\lambda|+r)^{\varphi_2}(s+r)^{2\tau}}. \label{eq:84}
\end{eqnarray}
We note that the finiteness of (\ref{eq:84}) is a consequence of (\ref{eq:83}) and (\ref{eq:82}). Substituting $t=\frac{r-|\omega|}{|\lambda|+|\omega|}$, we obtain for the integral in (\ref{eq:84}):
\begin{eqnarray}
&&  \int_{|\omega|}^\infty dr \frac{(r-|\omega|)^{p}r^{\varphi_1-1+\tau}}{(|\lambda|+r)^{\varphi_2}(s+r)^{2\tau}} \nonumber \\
&=& \frac{1}{(|\lambda|+|\omega|)^{\vphi_2-1-p}} \int_{0}^\infty dt \frac{t^{p}[(|\lambda|+|\omega|)t+|\omega|]^{\varphi_1-1+\tau}}{(t+1)^{\varphi_2}[(|\lambda|+|\omega|)t+|\omega| +s]^{2\tau}} \nonumber \\
&\geq& \frac{1}{(|\lambda|+|\omega|)^{\vphi_2-\vphi_1-p-\tau}} \int_{0}^\infty dt \frac{t^{p+\varphi_1-1+\tau}}{(t+1)^{\varphi_2}[(|\lambda|+|\omega|)t+|\omega| +s]^{2\tau}} \nonumber \\
&\geq&  \frac{C(\alpha,\beta,p,\tau)}{(|\lambda|+|\omega|)^{\vphi_2-\vphi_1-p-\tau}\max( |\lambda|+|\omega|, s + |\omega|)^{2\tau}}.\label{eq:85}
\end{eqnarray}
It remains to put together the information contained in (\ref{eq:82})-(\ref{eq:85}) and to evaluate the constants (for instance, $\varphi_2-\varphi_1-p-\tau=\frac{\eta_1+\eta_2}{2}+\eta_0-2\tau$).
\end{proof}

\section{Eigenvalue estimates - an operator theoretic approach}\label{sec:operator}

In this chapter we will present our second approach for studying the distribution of eigenvalues of non-selfadjoint operators, based on material from \cite{MR2836430}. As compared to the complex analysis
method this approach is quite elementary but, as we will see, still strong enough to improve upon some features of the former method.

\subsection{Kato's theorem}

The estimate we are going to present in Section \ref{subsec:num} will be a variant of the following classical estimate of Kato.
\begin{thm}[\cite{MR900507}]
Let $Z,Z_0 \in \bdd(\hil)$ be selfadjoint and assume that $Z-Z_0 \in \mathcal{S}_p(\hil)$ for some $p\geq 1$. Then there exist extended enumerations $\{z_j\}$ and $\{z_j^0\}$
of the discrete spectra of $Z$ and $Z_0$, respectively, such that
\begin{equation}\label{eq:Kato1}
  \sum_{j} |z_j-z_j^0|^p \leq \|Z-Z_0\|_{\mathcal{S}_p}^p.
\end{equation}
\end{thm}
Here an extended enumeration of the discrete spectrum is a sequence which contains all discrete eigenvalues, counting multiplicity, and which in addition might contain boundary points of the
essential spectrum. An immediate consequence of Kato's theorem is
\begin{cor}
Let $Z,Z_0 \in \bdd(\hil)$ be selfadjoint and assume that $Z-Z_0 \in \mathcal{S}_p(\hil)$ for some $p\geq 1$. Then
\begin{equation}\label{eq:Kato2}
  \sum_{\lambda \in \sigma_d(Z)} \dist(\lambda, \sigma(Z_0))^p \leq \|Z-Z_0\|_{\mathcal{S}_p}^p.
\end{equation}
\end{cor}
As it stands, Kato's theorem (and its corollary) need not be correct if (at least) one of the operators is non-selfadjoint. Indeed, even in the finite-dimensional case it can fail drastically.
\begin{example}
  Let $\hil= \mc^2$ and for $a>0$ define
$$ Z_0 = \left(
  \begin{array}{cc}
    0 & 1 \\
    0 & 0
  \end{array} \right), \quad Z = \left(
  \begin{array}{cc}
    0 & 1 \\
    a & 0
  \end{array} \right).$$
Then $\sigma_d(Z_0)=\{0\}$, $\sigma_d(Z)=\{\sqrt{a},-\sqrt{a}\}$, $\|Z-Z_0\|_{\mathcal{S}_p}^p = a^p$ and
$$ \sum_{\lambda \in \sigma_d(Z)} \dist(\lambda, \sigma_d(Z_0))^p = 2a^{p/2}.$$
Here for small $a$ the quotient of left- and right-hand side in (\ref{eq:Kato2}) can become arbitrarily large.
\end{example}
On the other hand, Kato's theorem is known to remain correct if $Z_0,Z$ and $Z-Z_0$ are normal \cite{MR1702207} or if $Z_0, Z$ and $Z-Z_0$ are unitary \cite{MR948353}, provided a multiplicative constant $\pi/2$ is added to the right-hand side. Inequality (\ref{eq:Kato2}) remains valid if $Z_0$ and $Z$ (but not necessarily $Z-Z_0$) are normal, but only if $p \geq 2$ \cite{MR577759}.
Moreover, the slightly weaker estimate
\begin{equation}\label{eq:Kato3}
  \sum_{\lambda \in \sigma_d(Z)} \dist(\lambda, \sigma(Z_0))^p \leq C_p \|Z-Z_0\|_{\mathcal{S}_p}^p,
\end{equation}
where the constant $C_p$ is independent of $Z_0$ and $Z$, holds provided that  $Z_0$ is selfadjoint and $Z$ is normal \cite{MR1314392}.

The case of most interest to us is the case where $Z_0$ is selfadjoint (and its spectrum is an interval) and $Z$ is arbitrary. In the next section we will show that in this case inequality (\ref{eq:Kato2}) does indeed remain correct. As we will see, this will be a simple corollary of a much more general estimate.
\begin{rem}
Recently it has been shown  \cite{Hansmann12} that for $p>1$ estimate (\ref{eq:Kato3}) remains valid if $Z_0$ is selfadjoint and $Z$ is arbitrary, even without the additional assumption that $\sigma(Z_0)$ is an interval. We will come back to this result in Chapter \ref{sec:outlook}.
\end{rem}

\subsection{An eigenvalue estimate involving the numerical range}\label{subsec:num}

The following theorem provides an estimate on the eigenvalues of $Z$ given the mere assumption that $Z-Z_0$ is in $\mathcal{S}_p(\hil)$. In particular, it does \emph{not} require that $Z_0$ is selfadjoint, normal or something alike.
\begin{thm}[\cite{MR2836430}] \label{thm:num1}
  Let $Z_0, Z \in \bdd(\hil)$ and assume that $Z-Z_0 \in \mathcal{S}_p(\hil)$ for some $p\geq 1$. Then
  \begin{equation}\label{eq:num1}
    \sum_{\lambda \in \sigma_d(Z)} \dist(\lambda, \num(Z_0))^p \leq \|Z-Z_0\|_{\mathcal{S}_p}^p.
  \end{equation}
\end{thm}
The proof of this theorem will be given below.
\begin{rem}
  It is interesting to observe that estimate (\ref{eq:num1}) remains valid for $p \in (0,1)$ if $Z_0$ and $Z$ are selfadjoint. This is in contrast to Kato's theorem, which will not be correct in this case. We refer to \cite{MR2836430} for a proof of these statements.
\end{rem}
Since the closure of the numerical range of a normal operator coincides with the convex hull of its spectrum, the following corollary is immediate.
\begin{cor}\label{cor:num}
  Let $Z_0, Z \in \bdd(\hil)$ and assume that $Z-Z_0 \in \mathcal{S}_p(\hil)$ for some $p\geq 1$. Moreover, let $Z_0$ be normal and assume that $\sigma(Z_0)$ is convex. Then
  \begin{equation}\label{eq:num2}
    \sum_{\lambda \in \sigma_d(Z)} \dist(\lambda, \sigma(Z_0))^p \leq \|Z-Z_0\|_{\mathcal{S}_p}^p.
  \end{equation}
\end{cor}
In particular, as mentioned above, this corollary applies if $Z_0$ is selfadjoint and the spectrum of $Z_0$ is an interval.
\begin{example}\label{ex:normal2}
Let us take a second look at Example \ref{ex:normal}, where $Z_0 \in \bdd(\hil)$ was normal with $\sigma(Z_0)=\sigma_{ess}(Z_0)=\overline{\md}$ and $Z=Z_0+M$ with $M \in \mathcal{S}_p(\hil)$ for some $p \geq 1$ (in particular, $\sigma_d(Z) \subset \overline{\md}^c$). The previous corollary then implies that
$$ \sum_{\lambda \in \sigma_d(Z)} (|\lambda|-1)^p \leq \|M\|_{\mathcal{S}_p}^p,$$
which is stronger than the corresponding estimate obtained in Example \ref{ex:normal} via the complex analysis approach.
\end{example}

\begin{rem}
A version of Corollary \ref{cor:num} for unbounded operators will be provided in Section \ref{subsec:num_unbound}.
\end{rem}

The proof of Theorem \ref{thm:num1} relies on the following characterization of Schatten-$p$-norms, see \cite{b_Simon05} Proposition 2.6.
\begin{lemma}\label{lem:schatten}
  Let $K \in \mathcal{S}_p(\hil)$, where $p \geq 1$. Then
$$ \|K\|_{\mathcal{S}_p}^p = \sup_{\{ e_i\}, \{f_i\}} \left\{ \sum_i |\langle K e_i, f_i \rangle |^p \right\},$$
where the supremum is taken with respect to arbitrary orthonormal sequences $\{e_i\}$ and $\{f_i\}$ in $\hil$.
\end{lemma}
\begin{proof}[Proof of Theorem \ref{thm:num1}]
Let $\Lambda=\{\lambda_1, \ldots, \lambda_n\}$ be an arbitrary finite subset of $\sigma_d(Z)$ and let
$$ P_Z(\Lambda)=P_Z(\lambda_1)+\ldots+P_Z(\lambda_n)$$
be the corresponding Riesz-Projection. Then $N:=\rank(P_Z(\Lambda))$ is the sum of the (algebraic) multiplicities of the $\lambda_i$'s and, invoking Schur's lemma, we can find an orthonormal basis $\{e_1,\ldots, e_N\}$ of $\ran(P_Z(\Lambda))$ such that
\begin{equation}
  \label{eq:schur}
 Ze_i = z_{i1}e_1 + z_{i2}e_2+ \ldots + z_{ii} e_i \qquad i = 1, \ldots, N,
\end{equation}
where the $z_{ii}$'s are the eigenvalues in $\Lambda$, counted according to their multiplicity (in other words, the finite-dimensional operator $Z|_{\ran(P_Z(\Lambda))}$ has upper-triangular form). Applying Lemma \ref{lem:schatten} to this particular sequence $\{e_i\}$ we obtain
\begin{eqnarray*}
  \|Z-Z_0\|_{\mathcal{S}_p}^p \geq \sum_{i=1}^N |\langle (Z-Z_0)e_i, e_i \rangle |^p = \sum_{i=1}^N |\langle Z e_i, e_i \rangle - \langle Z_0 e_i, e_i \rangle |^p.
\end{eqnarray*}
But $\langle Z e_i, e_i \rangle = z_{ii}$ and $\langle Z_0 e_i, e_i \rangle \in \num(Z_0)$, so the previous estimate implies that
$$ \sum_{\lambda \in  \Lambda} \dist(\lambda,\num(Z_0))^p \leq \|Z-Z_0\|_{\mathcal{S}_p}^p,$$
where each eigenvalue is counted according to its multiplicity. Noting that the right-hand side is independent of $\Lambda$ concludes the proof of Theorem \ref{thm:num1}.
\end{proof}

\begin{rem}\label{rem:Ouhabaz}
 The method of proof of Theorem \ref{thm:num1} can also be used to recover another recent result about the distribution of eigenvalues of non-selfadjoint operators, by Bruneau and Ouhabaz \cite[Theorem 1]{MR2455843}. Let $H$ be an $m$-sectorial operator in $\hil$ with the associated sesquilinear form $h(u,v)$ and let $\Re(H)$ denote the real part of $H$, i.e. the selfadjoint operator associated to the form $1/2 ( h(u,v)+ \overline{h(v,u)})$. For simplicity, let us suppose that $\dom(H) \subset \dom(\Re(H))$. Then, assuming that the negative part $\Re(H)_-$ of $\Re(H)$ is in $\mathcal{S}_p(\hil)$, we obtain as above that
$$ \sum_{i=1}^N |\langle \Re(H) e_i, e_i \rangle - \langle \Re(H)_+ e_i, e_i \rangle |^p \leq \|\Re(H)_-\|_{\mathcal{S}_p}^p,$$
where $\{e_i\}$ is a Schur Basis corresponding to a finite number of eigenvalues $(\lambda_i)_{i=1}^N$ of $H$ with $\Re(\lambda_i) < 0$.
Since $\langle \Re(H) e_i, e_i \rangle = \Re(\lambda_i)$ and $\langle \Re(H)_+ e_i, e_i \rangle \geq 0$ this estimate implies that
$\sum_{i=1}^N |\Re(\lambda_i)|^p \leq \|\Re(H)_-\|_{\mathcal{S}_p}^p$ and so we arrive at the estimate
\begin{equation}\label{eq:ouhabaz}
  \sum_{\lambda \in \sigma_d(H), \Re(\lambda)<0} |\Re(\lambda)|^p \leq \|\Re(H)_-\|_{\mathcal{S}_p}^p,
\end{equation}
the result of Bruneau and Ouhabaz.
\end{rem}

\subsection{Perturbations of non-negative operators}\label{subsec:num_unbound}

With the help of resolvents we can transfer the eigenvalue estimates of the previous section to unbounded operators. To make things simple, we will only study perturbations of non-negative operators.

\begin{thm}\label{thm:num_non}
Let $H_0 \in \cld(\hil)$ be selfadjoint with $\sigma(H_0) \subset [0,\infty)$. Let $H \in \cld(\hil)$ and assume that $a \in \rho(H) \cap (-\infty,0)$.  If $R_H(a)-R_{H_0}(a) \in \mathcal{S}_p(\hil)$ for some $p \geq 1$, then
\begin{equation}\label{eq:num_non}
\sum_{\lambda \in \sigma_d(H)} \frac{\dist(\lambda,[0,\infty))^p}{|\lambda+a|^p (|\lambda|+|a|)^p}  \leq 8^p \|R_H(a)-R_{H_0}(a)\|_{\mathcal{S}_p}^p.
\end{equation}
\end{thm}
\begin{rem}
  If we restrict the sum on the left-hand side of (\ref{eq:num_non}) to eigenvalues in the right half-plane, then the estimate remains valid without the additional constant $8^p$ on the right-hand side,
see \cite[Theorem 3.1]{MR2836430}.
\end{rem}
\begin{proof}
Applying Corollary \ref{cor:num} to $Z=R_H(a)$ and $Z_0=R_{H_0}(a)$ we obtain
$$ \sum_{\mu \in \sigma_d(R_H(a))} \dist(\mu, \sigma(R_{H_0}(a)))^p \leq \|R_H(a)-R_{H_0}(a)\|_{\mathcal{S}_p}^p.$$
The spectral mapping theorem and the assumption $\sigma(H_0) \subset [0,\infty)$ imply that $\sigma(R_{H_0}(a)) \subset [a^{-1},0]$, so applying the spectral mapping theorem again
we obtain
$$ \sum_{\lambda \in \sigma_d(H)} \dist((a-\lambda)^{-1}, [a^{-1},0])^p \leq \|R_H(a)-R_{H_0}(a)\|_{\mathcal{S}_p}^p.$$
All that remains is to observe that
$$ \dist\left((a-\lambda)^{-1} ,[a^{-1},0] \right) \geq \frac 1 8 \frac{\dist(\lambda,[0,\infty))}{|\lambda+a| (|\lambda|+|a|)},$$
see the proof of Theorem 3.3.1 in \cite{Hans_diss}.
\end{proof}
In applications to, for instance, Schr\"odinger operators, estimates on the  Schatten norm on the right-hand side of (\ref{eq:num_non}) will take a particular form, namely, we will have that
\begin{equation*}
  \|R_H(a)-R_{H_0}(a)\|_{\mathcal{S}_p}^p \leq C_0 |a|^{-\alpha} (|a|-|\omega|)^{-\beta}
\end{equation*}
for some constants $\alpha, \beta \geq 0$, $C_0>0$, $\omega < 0$ and every $a \in (-\infty,\omega)$ (compare this with (\ref{eq:70}) and (\ref{eq:77})). Note that $\alpha,\beta$ and $C_0$ may depend on $p$ but \textit{not} on $a$. In particular, in this case Theorem \ref{thm:num_non} provides us
with a whole family of estimates (i.e. one estimate for every $a < - \omega$) and we can take advantage of this fact by taking a suitable average of all these estimates, similar to what we have done in the derivation of Theorem \ref{cor:3}. This is the content of the
next theorem.

\begin{thm}\label{thm:num_non2}
Let $H_0 \in \cld(\hil)$ be selfadjoint with $\sigma(H_0) \subset [0,\infty)$ and let $H \in \cld(\hil)$ with  $(-\infty, \omega) \subset \rho(H)$ for some $\omega \leq 0$. Suppose that for some $p \geq 1$ there exist $\alpha, \beta \geq 0$ and $C_0>0$ such that for every $a< \omega$ we have
\begin{equation}
  \|R_H(a)-R_{H_0}(a)\|_{\mathcal{S}_p}^p \leq C_0 |a|^{-\alpha} (|a|-|\omega|)^{-\beta}.
\end{equation}
Then for every $\tau>0$ the following holds: If $\omega<0$ then
\begin{equation}\label{eq:num_non1}
\sum_{\lambda \in \sigma_d(H)} \frac{\dist(\lambda,[0,\infty))^p}{(|\lambda|+|\omega|)^{-\alpha-\beta+2p+\tau}} \leq C_0 C(\alpha,\beta,\tau,p) |\omega|^{-\tau}.
\end{equation}
If $\omega=0$ then for $s>0$
\begin{equation}
\sum_{\lambda \in \sigma_d(H),|\lambda|>s} \frac{\dist(\lambda,[0,\infty))^p}{|\lambda|^{-\alpha-\beta+2p+\tau}}
+ \sum_{\lambda \in \sigma_d(H), |\lambda|\leq s} \frac{\dist(\lambda,[0,\infty))^p}{|\lambda|^{-\alpha-\beta+2p-\tau}s^{2\tau}} \leq C_0 s^{-\tau} C(\alpha,\beta,\tau,p).
\end{equation}
\end{thm}
\begin{proof}
  From Theorem \ref{thm:num_non} and our assumption we obtain
$$ \sum_{\lambda \in \sigma_d(H)} \frac{\dist(\lambda,[0,\infty))^p}{|\lambda+a|^p (|\lambda|+|a|)^p}  \leq 8^p C_0 |a|^{-\alpha} (|a|-|\omega|)^{-\beta},$$
which we can rewrite as (also using the triangle inequality $|a+\lambda| \leq |a| + |\lambda|$)
$$ \sum_{\lambda \in \sigma_d(H)} \frac{\dist(\lambda,[0,\infty))^p |a|^{\alpha-1+\tau} (|a|-|\omega|)^\beta}{(|\lambda|+|a|)^{2p} (s+|a|)^{2\tau}}  \leq 8^p C_0 |a|^{-1+\tau} (s+|a|)^{-2\tau},$$
where we choose $s=0$ if $|\omega|>0$ and $s>0$ if $\omega=0$.
Integrating with respect to $r:=|a|$ we obtain
\begin{equation*}
 \sum_{\lambda \in \sigma_d(H)} \dist(\lambda, [0,\infty))^p \int_{|\omega|}^\infty dr \frac{r^{\alpha-1+\tau} (r-|\omega|)^\beta}{(|\lambda|+r)^{2p} (s+r)^{2\tau}}
\leq C_0 8^p \int_{|\omega|}^\infty dr \frac{1}{ r^{1-\tau} (s+r)^{2\tau}}.
\end{equation*}
As in the proof of Theorem \ref{cor:3} we can estimate the integral on the left from below by
 $$ \frac{C(\alpha,\beta,p,\tau)}{(|\lambda|+|\omega|)^{-\alpha-\beta+2p-\tau} \max(|\lambda|+|\omega|,s+|\omega|)^{2\tau}},$$
and the integral on the right is equal to
$$    \left\{  \begin{array}{cl}
      \frac 1 {\tau |\omega|^\tau}, & |\omega| > 0 \text{ and } s = 0 \\[4pt]
      \frac{C(\tau)}{s^\tau}, & \omega = 0 \text{ and } s > 0.
    \end{array} \right.
$$
Putting everything together concludes the proof.
\end{proof}

\section{Comparing the two approaches}\label{sec:comparison}

\setcounter{equation}{0}
\numberwithin{equation}{section}

Above we have developed two quite different approaches for obtaining inequalities involving the eigenvalues of non-selfadjoint operators.
One is based on applying complex-analysis theorems on the distribution of zeros to perturbation determinants (Chapter \ref{sec:complex}), and the other relies on direct operator-theoretic arguments
involving the numerical range (Chapter \ref{sec:operator}). We now wish to compare the results obtained by the two methods, in order to understand the strengths and limitations of
each approach.

We consider only the case in which $A_0$ is a bounded self-adjoint operator, with $\sigma(A_0)=[a,b]$, and $A=A_0+M$, where $M\in \mathcal{S}_p(\hil)$.
Corollary \ref{cor:2}, obtained using the complex-analysis approach, tells us that, when $p\geq 1$ and for any $\tau>0$ we have
\begin{equation}
\label{eq:c1}
  \sum_{\lambda \in \sigma_d(A)} \frac{\dist(\lambda,[a,b])^{p+1+\tau}}{|\lambda-b||\lambda-a|} \leq C(p,\tau) (b-a)^{-1+\tau} \|M\|_{\mathcal{S}_p}^p
\end{equation}
and that when $p<1$, and $0<\tau<1-p$,
\begin{equation}
\label{eq:c3}
  \sum_{\lambda \in \sigma_d(A)} \left( \frac{\dist(\lambda,[a,b])}{|b-\lambda|^{1/2}|a-\lambda|^{1/2}}\right)^{p+1+\tau} \leq C(p,\tau) (b-a)^{-p} \|M\|_{\mathcal{S}_p}^p.
\end{equation}
Corollary \ref{cor:num}, obtained using the numerical range approach, tells us that
\begin{equation}\label{eq:c2}
    \sum_{\lambda \in \sigma_d(A)} \dist(\lambda, [a,b])^p \leq \|M\|_{\mathcal{S}_p}^p.
\end{equation}
Clearly a good feature of (\ref{eq:c2}), as opposed to (\ref{eq:c1}), is the absence of a constant $C$ on the right-hand side. An optimal
value of the constant $C(p,\tau)$ in (\ref{eq:c1}) is not known, and though explicit upper bounds for such an optimal value could be
extracted by making all the estimates used in its derivation explicit, the resulting expression would be complicated, and there is little reason to
expect that it would yield a sharp result.

We now compare the information that can be deduced from these inequalities  regarding the asymptotic behavior of sequences of eigenvalues.

(i) Assume first that $p\geq 1$. To begin with consider a sequence of eigenvalues $\{\lambda_k\}$ with $\lambda_k\rightarrow \lambda^*\in (a,b)$ as $k\rightarrow \infty$.
Then $|\lambda_k-a|$ and $|\lambda_k-b|$ are bounded from below by some positive constant, hence we conclude from (\ref{eq:c1}) that the
sum $\sum_{k=1}^\infty \dist(\lambda_k,[a,b])^{p+1+\tau}$ is finite, for any $\tau>0$. However, (\ref{eq:c2}) implies the finiteness
of $\sum_{k=1}^\infty \dist(\lambda_k,[a,b])^{p}$, obviously a stronger result.

If we consider a sequence $\{\lambda_k\}$ with $\lambda_k\rightarrow  a$, then (\ref{eq:c1}) implies
$$\sum_{k=1}^\infty \frac{\dist(\lambda_k,[a,b])^{p+1+\tau}}{|a-\lambda_k|}<\infty,  $$
for any $\tau>0$. However, since $|\lambda_k-a| \geq \dist(\lambda_k,[a,b])$, so that
$$\dist(\lambda_k, [a,b])^p\geq \frac{\dist(\lambda_k, [a,b])^{p+1}}{|\lambda_k-a|}$$
(\ref{eq:c2}) implies the stronger result
$$\sum_{k=1}^\infty \frac{\dist(\lambda_k,[a,b])^{p+1}}{|\lambda_k-a|}<\infty.$$
Thus, we have established the superiority of (\ref{eq:c2}) over (\ref{eq:c1}).

(ii) Let us examine the case $0<p<1$,  $0<\tau<1-p$.
Corollary \ref{cor:num} is not valid for $p<1$, but we can use the fact that $\mathcal{S}_p(\hil)\subset \mathcal{S}_1(\hil)$ to conclude that
\begin{equation}\label{eq:c4}
    \sum_{\lambda \in \sigma_d(A)} \dist(\lambda, [a,b]) \leq \|M\|_{\mathcal{S}_1}.
\end{equation}
Considering a sequence $\{\lambda_k\}$ of eigenvalues with $\lambda_k\rightarrow \lambda^*\in (a,b)$ as $k\rightarrow \infty$, (\ref{eq:c3}) implies
$\sum_{k=1}^\infty \dist(\lambda_k,[a,b])^{p+1+\tau}<\infty$, which is weaker than the result $$\sum_{k=1}^\infty \dist(\lambda_k,[a,b])<\infty$$ implied by (\ref{eq:c4}).

However, considering a sequence $\{\lambda_k\}$ of eigenvalues with $\lambda_k\rightarrow  a$ as $k\rightarrow \infty$, (\ref{eq:c3}) gives
\begin{equation}\label{eq:c5}\sum_{k=1}^\infty \Big(\frac{\dist(\lambda_k,[a,b])}{|\lambda_k-a|^{\frac{1}{2}}}\Big)^{p+1+\tau}<\infty.\end{equation}
This result does {\it{not}} follow from (\ref{eq:c4}). To see this, take a real sequence with $\lambda_k<a$, so that $|\lambda_k-a|=\dist(\lambda_k,[a,b])$. Then
(\ref{eq:c5}) becomes $$\sum_{k=1}^{\infty}\dist(\lambda_k,[a,b])^{\frac{1}{2}(p+1+\tau)}<\infty,$$ which is stronger than the result given by
(\ref{eq:c4}) since $p+\tau<1$ implies that $\frac{1}{2}(p+1+\tau)<1$.

Summing up, we have seen that in nearly all cases Corollary \ref{cor:num}, proved by the numerical range approach, provides sharper information on the asymptotics of eigenvalue sequences
than provided by Corollary \ref{cor:2}, proved by the complex analysis approach, the sole exception being the case $p<1$ when considering a sequence of eigenvalues
converging to an edge of the essential spectrum.

This, however, is not the end of the story. Corollary \ref{cor:2} which we have been discussing, is only the simplest result that we can obtain using the
complex analysis approach. We recall that Theorem \ref{thm:3}, which provides inequalities on the
eigenvalues assuming an estimate on the quantity $ \| M_2R_{A_0}(\lambda)M_1\|_{\mathcal{S}_p}^p$, where $M_1,M_2$ are a pair of operators with $M_1M_2=M$.
Corollary \ref{cor:2} was obtained by taking $M_1=I,M_2=M$. As we shall see in Chapter \ref{sec:jacobi}, in an application to Jacobi operators, by choosing a
different decomposition $M=M_1M_2$ one may use Theorem \ref{thm:3} to obtain stronger results than those provided by Corollaries \ref{cor:2} and \ref{cor:num}.

We may thus conclude that if one is considering a general bounded operator of the form $A=A_0+M$, with $A_0$ selfadjoint, $\sigma(A_0)=[a,b]$, and
$M\in \mathcal{S}_p(\hil)$, $p\geq 1$, and the only quantitative information available is a bound on the norm $ \| M\|_{\mathcal{S}_p}$, then the best
estimate on the discrete spectrum of $A$ is provided by the numerical range method. If, however, one is dealing with specific classes of operators of the
above form which have a special structure which allows to perform an appropriate decomposition $M=M_1M_2$ and estimate  $ \| M_2R_{A_0}(\lambda)M_1\|_{\mathcal{S}_p}^p$,
one can sometimes obtain stronger results using the complex analysis approach (Theorem \ref{thm:3}).

\begin{rem*}
What has been said in the last paragraph applies also to the case of unbounded operators, as we will see in our discussion of Schr\"odinger operators in Chapter \ref{sec:schroedinger}.
\end{rem*}

\section{Applications}\label{sec:applications}
\numberwithin{equation}{subsection}

In this chapter we will finally apply our abstract estimates to some more concrete situations. Namely, we will analyze the discrete eigenvalues of  bounded Jacobi operators on $l^2(\mz)$ and of unbounded Schr\"odinger operators in $L^2(\mr^d)$, respectively.

\subsection{Jacobi operators}\label{sec:jacobi}

In this section, which is based on \cite{HK09}, we apply our general results on bounded non-selfadjoint perturbations of selfadjoint operators
 to obtain estimates on the discrete spectrum of complex Jacobi operators.

The  spectral theory of Jacobi operators is a classical subject with many beautiful results, though by far the
majority of results relate to selfadjoint Jacobi operators. Using our results, we are able to obtain new estimates on the eigenvalues of non-selfadjoint  Jacobi operators which are nearly as strong as those which have been obtained in the selfadjoint case. The techniques, however, are very different, as previous results for the
selfadjoint case have been obtained by methods which rely very strongly on the selfadjointness. This example thus gives a striking illustration of
the utility of our general results in studying a concrete class of operators.

Another interesting feature of these results is that they provide an example in which the
 results proved by means of the complex-analysis approach of Chapter \ref{sec:complex} are (in many respects) stronger
 than those we can obtain at present using the operator-theoretic approach of Chapter \ref{sec:operator}.
 This is in contrast with the case of `general' operators, for which, as we have discussed above, the operator-theory approach provides results which are usually stronger.\\

\noindent Given three bounded complex sequences $\{a_k\}_{k \in \mz}, \{b_k\}_{k \in \mz}$ and $ \{c_k\}_{k\in \mz}$, we define the associated (complex) \textit{Jacobi operator} $J=J(a_k,b_k,c_k): l^2(\mz) \to l^2(\mz)$ as follows:
\begin{equation}\label{eq:102}
(J u)(k)=a_{k-1}u(k-1)+b_k u(k)+c_k u(k+1), \quad u \in l^2(\mz).
\end{equation}
It is easy to see that $J$ is a bounded operator on $l^2(\mz)$ with $$\|J\| \leq \sup_k |a_k| + \sup_k |b_k| + \sup_k|c_k|.$$
Moreover, with respect to the standard basis $\{\delta_k\}_{k \in \mz}$ of $l^2(\mz)$, i.e., $\delta_k(j)=0$ if $j\neq k$ and $\delta_k(k)=1$, $J$ can be represented by the two-sided infinite tridiagonal matrix
$$\left(
      \begin{array}{ccccccc}
       \ddots & \ddots & \ddots &  &  &  &  \\
        &  a_{-1} & b_0 & c_0 &  &  &   \\
        &  & a_0 & b_1 & c_1 &  &   \\
        &  &  & a_1 & b_2 & c_2 &  \\
        &  &  &  & \ddots & \ddots & \ddots
      \end{array}
    \right).
$$
In view of this representation it is also customary to refer to $J$ as a \textit{Jacobi matrix}.
\begin{example}
The \textit{discrete Laplace operator} on $l^2(\mz)$ coincides with the Jacobi operator $J(1,-2,1)$. Similarly, the Jacobi operator $J(-1,2+d_k,-1)$ ($d_k\in \mc$) describes a \textit{discrete Schr\"odinger operator}.
\end{example}

\noindent In the following, we will focus on Jacobi operators  which are perturbations of the \textit{free} Jacobi operator $J_0=J(1,0,1)$, i.e.,
\begin{equation}
  \label{eq:103}
(J_0u)(k)=u(k-1)+u(k+1), \quad u \in l^2(\mz).
\end{equation}
More precisely, if $J=J(a_k,b_k,c_k)$ is defined as above, then throughout this section we assume that $J-J_0$ is compact.
\begin{prop}
  The operator $J-J_0$ is compact if and only if
\begin{equation}\label{eq:112}
 \lim_{|k|\rightarrow \infty}a_k=\lim_{|k|\rightarrow
\infty}c_k=1 \quad \text{and} \quad \lim_{|k|\rightarrow
\infty}b_k=0.
\end{equation}
\end{prop}
\begin{proof}
  It is easy to see that $J-J_0$ is a norm limit of finite rank operators, and hence compact, if (\ref{eq:112}) is satisfied. On the other hand, if $J-J_0$ is compact then it maps weakly convergent zero-sequences into norm convergent zero-sequences. In particular,
$$ \|(J-J_0)\delta_k\|_{l^2}^2 = |a_k-1|^2+|b_k|^2+|c_{k-1}-1|^2 \overset{|k| \to \infty}{\longrightarrow} 0$$
as desired.
\end{proof}
Let $F:l^2(\mz) \to L^2(0,2\pi)$ denote the Fourier transform, i.e.,
$$({F}u)(\theta)= \frac 1 {\sqrt{2\pi}} \sum_{k \in \mz} e^{ik\theta}u_k.$$
Then  for $u \in l^2(\mz)$ and $\theta \in
[0,2\pi)$ we have
\begin{equation}\label{fourier}
({F}J_0 u)(\theta) = 2\cos(\theta) ({F}u)(\theta),
\end{equation}
as a short computation shows. In particular, we see  that $J_0$ is unitarily equivalent to the operator of multiplication by the function $2\cos(\theta)$ on $L^2(0,2\pi)$, and so the spectrum of $J_0$ coincides with the interval $[-2,2]$. Consequently, the compactness of $J-J_0$ implies that
$$ \sigma(J)=[-2,2] \dotcup \sigma_d(J),$$
i.e., the isolated eigenvalues of $J$ are situated in $\mc \setminus [-2,2]$ and can accumulate on $[-2,2]$ only.

Our aim is to derive estimates on $\sigma_d(J)$ given the stronger assumption that $J-J_0 \in \mathcal{S}_p$ (for simplicity, in this section we set  $\mathcal{S}_p=\mathcal{S}_p(l^2(\mz)$). To this end, let us define a sequence $v= \{v_k\}_{k \in \mz}$ by setting
\begin{equation}\label{defdk}v_k=\max\Big(|a_{k-1}-1|,|a_{k}-1|,|b_k|,|c_{k-1}-1|,|c_k-1|\Big).\end{equation}
Clearly, the compactness of $J-J_0$ is equivalent to $v_k$
converging to $0$. Moreover, for $p\geq 1$ we will show in Lemma \ref{esti1} below that $J-J_0 \in \mathcal{S}_p$ if and only if $v \in l^p(\mz)$, and the $\mathcal{S}_p$-norm of $J-J_0$ and the $l^p$-norm of $v$ are equivalent.
If $p \in (0,1)$, then the $\mathcal{S}_p$-norm of $J-J_0$ and the $l^p$-norm of $v$ are still equivalent in the diagonal case when $a_k=c_k\equiv 1$. In general, however, we only obtain a one-sided estimate.
\begin{lemma}[\cite{HK09}, Lemma 8]\label{esti1}
  Let $p > 0$. Then
  \begin{equation}
     \|J-J_0\|_{\mathcal{S}_p} \leq 3 \|v\|_{l^p}.
  \end{equation}
Moreover, if $p \geq 1$ then
\begin{equation}
  \label{eq:114}
 6^{-{1}/{p}} \|v\|_{l^p}\leq \|J-J_0\|_{\mathcal{S}_p}.
\end{equation}
\end{lemma}
From the above estimate and Corollary \ref{cor:2} we obtain
\begin{thm}\label{thm:12}
Let $\tau \in (0,1)$. If $p \geq 1-\tau$ then
\begin{equation}\label{eq:108}
\sum_{\lambda \in \sigma_{d}(J)}
\frac{\dist(\lambda,[-2,2])^{p+1+\tau}}{|\lambda^2-4|} \leq
C(\tau,p) \|v\|_{l_p}^p.
\end{equation}
Moreover, if $p \in (0,1-\tau)$ then
\begin{equation}\label{eq:111}
\sum_{\lambda \in \sigma_{d}(J)} \left(\frac{\dist(\lambda,[-2,2])}{|\lambda^2-4|^{{1}/{2}}}\right)^{p+1+\tau} \leq
C(\tau,p) \|v\|_{l_p}^p.
\end{equation}
\end{thm}
\begin{rem}
  A slightly weaker version of the previous theorem has first been obtained by Golinskii, Borichev and Kupin, compare Remark \ref{rem:borichev}.
\end{rem}
In addition, Corollary \ref{cor:num} implies
\begin{thm}\label{thm:12b}
If $p \geq 1$ then
\begin{equation}
\sum_{\lambda \in \sigma_{d}(J)}
\dist(\lambda,[-2,2])^{p} \leq \|v\|_{l_p}^p.
\end{equation}
\end{thm}

As was already discussed in Chapter \ref{sec:comparison}, the result of Theorem \ref{thm:12b} is stronger than that of Theorem \ref{thm:12} in the case
$p\geq 1$. However, we now show that both of these results can be considerably improved, when $p\geq 1$, by a more refined application of Theorem \ref{thm:3}.
The following theorem is our main result on the discrete eigenvalues of Jacobi operators. Its proof will be presented below.

\begin{thm}\label{THM2}
Let $\tau \in (0,1)$. If $v \in l^p(\mz)$, where $p > 1$, then
\begin{equation}\label{inequality2}
\sum_{\lambda \in \sigma_d(J)}
\frac{\dist(\lambda,[-2,2])^{p+\tau}}{|\lambda^2-4|^{{1}/{2}}}
\leq C(p,\tau)\|v\|_{l^p}^p.
\end{equation}
Furthermore, if $v \in l^1(\mz)$ then
\begin{equation}\label{inequality2b}
\sum_{\lambda \in \sigma_d(J)}
\frac{\dist(\lambda,[-2,2])^{1+\tau}}{|\lambda^2-4|^{\frac{1}{2}+\frac
{\tau}{4}}} \leq C(\tau)\|v\|_{l^1}.
\end{equation}
\end{thm}

\noindent Let us compare the previous theorem with Theorem \ref{thm:12} and \ref{thm:12b}, respectively. To begin, we note that a direct calculation shows that for $\tau >0$, $\lambda \in \mc \setminus [-2,2]$ and $p>1$ we have
\begin{equation*}
  \frac{\dist(\lambda,[-2,2])^{p+1+\tau}}{|\lambda^2-4|} \leq \frac{\dist(\lambda,[-2,2])^{p+\tau}}{|\lambda^2-4|^{{1}/{2}}}.
\end{equation*}
Moreover, if $\lambda \in \mc \setminus [-2,2]$ and $|\lambda| \leq \|J\|$, then
\begin{equation*}
  \frac{\dist(\lambda,[-2,2])^{2+\tau}}{|\lambda^2-4|} \leq C(\tau, \|J\|) \frac{\dist(\lambda,[-2,2])^{1+\tau}}{|\lambda^2-4|^{\frac{1}{2}+\frac{\tau}{4}}}.
\end{equation*}
Hence, inequalities (\ref{inequality2}) and (\ref{inequality2b}) provide more information on the discrete spectrum of $J$ than inequality (\ref{eq:108}), i.e., Theorem \ref{THM2} is stronger  than Theorem \ref{thm:12}.

The advantage of Theorem \ref{THM2} over Theorem \ref{thm:12b} can be seen by considering sequences of eigenvalues
$\{\lambda_k\}$ converging to an endpoint of the spectrum. If $\lambda_k\rightarrow 2$ as $k\rightarrow \infty$, Theorem \ref{thm:12b} implies the
convergence of the sum $\sum_{k=1}^\infty |\lambda_k-2|^p$, while Theorem \ref{THM2} implies the convergence of the sum $\sum_{k=1}^\infty |\lambda_k-2|^{p-\frac{1}{2}+\tau}$, which is strictly stronger when $\tau<\frac{1}{2}$.

It should be noted, however, that Theorem \ref{THM2} does not subsume Theorem \ref{thm:12b}, since for sequences $\lambda_k\rightarrow (-2,2)$, Theorem
\ref{THM2} only implies the convergence of   $\sum_{k=1}^\infty |\lambda_k-2|^{p+\tau}$ for any $\tau>0$, which is weaker than the convergence of
 $\sum_{k=1}^\infty |\lambda_k-2|^p$ given by Theorem \ref{thm:12b}.

\begin{problem}\label{p1}
In view of the previous discussion it is natural to conjecture that a result implying both Theorem \ref{thm:12b} and \ref{THM2} is true, namely the inequality obtained
by setting $\tau=0$ in (\ref{inequality2}):
\begin{equation}\label{eq:conjecture}\sum_{\lambda \in \sigma_d(J)}
\frac{\dist(\lambda,[-2,2])^{p}}{|\lambda^2-4|^{{1}/{2}}}
\leq C(p)\|v\|_{l^p}^p.\end{equation}
However, we have not been able to prove such a result, and it remains an open question.
We note here that in the case of selfadjoint $J$, (\ref{eq:conjecture}) was proved by Hundertmark and Simon \cite{Hundertmark02}. It was then partly extended to the non-selfadjoint case by
Golinskii and Kupin \cite{Golinskii07}, who considered eigenvalues outside a diamond-shaped region avoiding the interval $[-2,2]$. We can therefore
consider Theorem \ref{THM2} as a near-generalization of the results of Hundertmark and Simon (and Golinskii and Kupin), and it would be interesting to understand whether the gap
between our result (with $\tau>0$) and their results  ($\tau=0$) can be closed for general non-selfadjoint Jacobi operators.
\end{problem}

\begin{proof}[Proof of Theorem \ref{THM2}]

 Some of the technical results needed will be quoted without proofs, for which we will refer to \cite{HK09}.

Let the multiplication operator $M_v \in \bdd(l^2(\mz))$ be defined by $M_v\delta_k=v_k \delta_k$, where
the  sequence $v=\{v_k\}$ was defined in (\ref{defdk}).
Furthermore, we define the operator $U \in \bdd(l^2(\mz))$ by setting
$$U \delta_k=u_k^- \delta_{k-1} +
u_k^0\delta_k +u_k^+\delta_{k+1},$$ where (using the convention that $\frac{0}{0}=1$)
$$u_k^-=\frac{c_{k-1}-1}{\sqrt{v_{k-1}v_k}},\quad
u_k^0=\frac{b_k}{v_k} \quad \text{ and } \quad
u_k^+=\frac{a_{k}-1}{\sqrt{v_{k+1}v_k}}.$$ 
It is then easily checked that
\begin{equation}\label{rj}J-J_0=M_{v^{1/2}}UM_{v^{1/2}},\end{equation}
where $v^{1/2}=\{ v_k^{1/2} \}$.
Moreover, the definition of $\{v_k\}$ implies that
$$|u_k^-|\leq 1, \quad |u_k^0|\leq 1 \quad \text{and} \quad |u_k^+|\leq1,$$
showing  that $\|U\|\leq 3$.

We intend to prove Theorem \ref{THM2} by an application of Theorem \ref{thm:3}. Since we have seen above that $J-J_0=M_{v^{1/2}}UM_{v^{1/2}}$, we will apply that theorem choosing (with the notation of that theorem) $M_1=M_{v^{1/2}}$ and $M_2=UM_{v^{1/2}}$, and so we need an appropriate bound on the Schatten norm of  $UM_{v^{1/2}}R_{J_0}(\lambda)M_{v^{1/2}}$.
\begin{lemma}\label{lem:est_g}
Let $v \in l^p(\mz)$, where $p \geq 1$. Then the following holds: if $p>1$ then
\begin{equation}\label{Q2}
   \|UM_{v^{1/2}}R_{J_0}(\lambda)M_{v^{1/2}}\|_{\mathcal{S}_p}^p \leq  \frac{C(p) \|v\|_{l^p}^p}{\dist(\lambda,[-2,2])^{p-1}|\lambda^2-4|^{1/2}}.
\end{equation}
Furthermore, if $v \in l^1(\mz)$, then  for every $\eps \in (0,1)$ we have
\begin{equation}\label{Q3}
   \|UM_{v^{1/2}}R_{J_0}(\lambda)M_{v^{1/2}}\|_{\mathcal{S}_1} \leq \frac{C(\eps)\|v\|_{l^1} }{\dist(\lambda,[-2,2])^{\eps}|\lambda^2-4|^{{(1-\eps)}/2}}.
\end{equation}
\end{lemma}
The proof of Lemma \ref{lem:est_g} will be given below. First, let
us continue with the proof of Theorem \ref{THM2}. To this end, let us assume that $v \in l^p(\mz)$ and let us fix $\tau \in (0,1)$. Considering the case $p>1$ first, we obtain from (\ref{Q2}) and Theorem \ref{thm:3}, with $\alpha=p-1$,
$\beta = - 1/2$ and $K=C(p)\|v\|_{l^p}^p$, i.e., $\eta_1=p+\tau$ and $\eta_2=p-1+\tau$,
\begin{equation*}
\sum_{\lambda \in \sigma_d(J)}
\frac{\dist(\lambda,[-2,2])^{p+\tau}}{|\lambda^2-4|^{{1}/{2}}}
\leq C(p,\tau)\|v\|_{l^p}^p.
\end{equation*}
Similarly, if $p=1$, then we obtain from (\ref{Q3}) and Theorem \ref{thm:3} that, for $\eps \in (0,1)$ and $\tilde{\tau} \in (0,1)$,
\begin{equation*}
  \sum_{\lambda \in \sigma_d(J)}
\frac{\dist(\lambda,[-2,2])^{1+\eps+\tilde{\tau}}}{|\lambda^2-4|^{{(1+\eps)}/2}} \leq C(\tilde\tau,\eps)\|v\|_{l^1}.
\end{equation*}
Choosing $\eps=\tilde{\tau}=\tau/2$ concludes the proof of
Theorem \ref{THM2}.
\end{proof}

It remains to prove Lemma \ref{lem:est_g}. In the following, let $v \in l^p(\mz)$ where $p \geq 1$. To begin,  we recall (see (\ref{fourier})) that
$$({F}J_0 f)(\theta) = 2\cos(\theta) ({F}f)(\theta), \quad f \in l^2(\mz), \quad \theta \in [0,2\pi),$$
where ${F}$ denotes the Fourier transform. Consequently, for $\lambda \in \mc\setminus [-2,2]$ we have
  \begin{eqnarray*}
    R_{J_0}(\lambda)= {F}^{-1}M_{g_\lambda}{F},
  \end{eqnarray*}
where $M_{g_\lambda} \in \bdd(L^2(0,2\pi))$ is the operator of
multiplication by the bounded function
\begin{equation}
  g_\lambda(\theta)= (\lambda-2\cos(\theta))^{-1}, \quad \theta \in [0,2\pi).\label{Q4}
\end{equation}
Since $g_\lambda= |g_\lambda|^{1/2} \cdot \frac{g_\lambda}{|g_\lambda|}\cdot |g_\lambda|^{1/2},$ we can define the unitary operator $T=F^{-1}M_{g_\lambda/|g_\lambda|}F$ to obtain the identity
\begin{equation*}
  \|UM_{v^{1/2}}R_{J_0}(\lambda)M_{v^{1/2}}\|_{\mathcal{S}_p}^p =   \|UM_{v^{1/2}}{F}^{-1} M_{|g_\lambda|^{1/2}}{F}T{F}^{-1} M_{|g_\lambda|^{1/2}}{F}M_{v^{1/2}}\|_{\mathcal{S}_p}^p.
\end{equation*}
Using H\"older's inequality for Schatten norms (see Section \ref{sec:schatten}), and recalling that $\|U\| \leq 3$, we thus obtain
\begin{eqnarray}
 \|UM_{v^{1/2}}R_{J_0}(\lambda)M_{v^{1/2}}\|_{\mathcal{S}_p}^p &\leq& 3^{p} \|M_{v^{1/2}}{F}^{-1} M_{|g_\lambda|^{1/2}}{F}\|_{\mathcal{S}_{2p}}^p \|{F}^{-1} M_{|g_\lambda|^{1/2}}{F}M_{v^{1/2}}\|_{\mathcal{S}_{2p}}^p \nonumber \\
&=& 3^p \|M_{v^{1/2}}{F}^{-1} M_{|g_\lambda|^{1/2}}{F}\|_{\mathcal{S}_{2p}}^{2p}.  \label{R1}
\end{eqnarray}
For the last identity we used the selfadjointness of the bounded operators  $M_{v^{{1/2}}}$ and $F^{-1}M_{|g_\lambda|^{1/2}}F$  and the fact that the Schatten norm of an operator and its adjoint coincide.

To derive an estimate on the Schatten norm on the right-hand side of (\ref{R1}), we
will use the following lemma (see \cite{HK09}, Lemma 10).
Here, as above, $M_u \in \bdd(l^2(\mz))$ and $M_h \in \bdd(L^2(0,2\pi))$ denote the operators of multiplication by a sequence $u=\{u_m\} \in l^\infty(\mz)$ and a function $h \in L^\infty(0,2\pi)$, respectively.
\begin{lemma} \label{lem:simon}
Let $q \geq 2$ and suppose that  $u=\{u_m\} \in l^q(\mz)$ and $h \in
L^\infty(0,2\pi)$. Then
\begin{equation}
  \|M_u {F}^{-1} M_h {F} \|_{\mathcal{S}_q} \leq (2\pi)^{-1/q} \|u\|_{l^q} \|h\|_{L^q}.
\end{equation}
\end{lemma}
\begin{rem}
For operators on $L^2(\mr^d)$ an analogous result is well-known, see Lemma \ref{lem:simon2} below.
\end{rem}
Since $p \geq 1$ (and so $2p \geq 2$), the previous lemma and (\ref{R1}) imply that
\begin{equation}
  \label{eq:123}
\| UM_{v^{1/2}}R_{J_0}(\lambda)M_{v^{1/2}}\|_{\mathcal{S}_p}^p \leq C(p) \|M_{v^{1/2}}{F}^{-1} M_{|g_\lambda|^{1/2}}{F}\|_{\mathcal{S}_{2p}}^{2p} \leq C(p) \|v\|_{l^p}^p\|g_\lambda\|_{L^p}^p.
\end{equation}
The proof of Lemma \ref{lem:est_g} is completed by an application of
the following result (\cite{HK09}, Lemma 9).
\begin{lemma}\label{lem:resolv1}
  Let $\lambda \in \mc \setminus [-2,2]$ and let $g_\lambda: [0,2\pi) \to \mc$ be defined by (\ref{Q4}). Then the following holds: If $p>1$ then
  \begin{equation}\label{eq:resolv1}
    \| g_\lambda \|_{L^p}^p \leq \frac{C(p)}{\dist(\lambda,[-2,2])^{p-1}|\lambda^2-4|^{1/2 }}.
  \end{equation}
Furthermore, for every $0 < \eps < 1$ we have
\begin{equation}\label{eq:resolv2}
     \| g_\lambda \|_{L^1} \leq \frac{C(\eps)}{\dist(\lambda,[-2,2])^{\eps}|\lambda^2-4|^{{(1-\eps)}/ 2 }}.
\end{equation}
\end{lemma}

\begin{rem}
  In this section we have finally seen why it was advantageous to formulate Theorem \ref{thm:3} in terms of estimates on $M_2R_{A_0}(\lambda)M_1$ (see (\ref{eq:42})) instead of estimates on $MR_{A_0}(\lambda)$ (where $A=A_0+M=A_0+M_1M_2$). Without this decomposition  the estimates in Theorem \ref{THM2} could have been proved for $p \geq 2$ only, due to the restriction to such $p$'s in Lemma \ref{lem:simon}.
\end{rem}

\subsection{Schr\"odinger operators}\label{sec:schroedinger}

In the following we consider Schr\"odinger operators $H=-\Delta + V$ in $L^2(\mr^d)$, where $V \in L^p(\mr^d)$ is a complex-valued potential with
\begin{equation}
  \label{eq:01}
\begin{array}{cl}
  p \geq 1, & \text{if } d =1 \\
  p > 1, & \text{if } d = 2 \\
  p \geq \frac d 2, & \text{if } d \geq 3.
\end{array}
\end{equation}
More precisely, $H$ is the unique $m$-sectorial operator associated to the closed, densely defined, sectorial form
$$\mathcal{E}(f,g)= \langle \nabla f, \nabla g  \rangle + \langle Vf, g \rangle, \quad \dom(\mathcal{E})=W^{1,2}(\mr^d).$$
In particular, there exists $\omega \leq 0$ and $\theta \in [0, \frac \pi 2)$ such that
\begin{equation}
  \label{eq:num_omega}
\sigma(H) \subset \overline{\num}(H) \subset \{ \lambda : |\arg(\lambda-\omega)| \leq \theta \}
\end{equation}
and so (\ref{sec1:num}) implies that
\begin{equation}
  \label{eq:reso_es}
\|R_H(\lambda)\| \leq |\Re(\lambda)-\omega|^{-1}, \qquad \Re(\lambda) < \omega.
\end{equation}
\begin{rem}\label{rem_relcomp}
  We note that for $V \in L^p(\mr^d)$ with $p \geq 2$ if $d \leq 3$ and $p > d/2$ if $d \geq 4$ the multiplication operator $M_V$, defined as
$$ (M_V f)(x) = V(x) f(x), \quad \dom(M_V)=\{ f \in L^2 : Vf \in L^2 \},$$
is relatively compact with respect to $-\Delta$ (see Lemma \ref{lem:8}), so in this case the operator $H$ coincides with the usual operator sum $-\Delta+ M_V$ defined on $\dom(-\Delta)=W^{2,2}(\mr^d)$.
Here, as usual, $-\Delta$ is defined via the Fourier transform $F$ on $L^2$, i.e. $-\Delta=F^{-1} M_{|k|^2} F$.
\end{rem}
It can be shown that the resolvent difference $(-a-H)^{-1}-(-a+\Delta)^{-1}$ is compact for $a>0$ sufficiently large, so Corollary \ref{prop:4} implies that the spectrum of $H$ consists of $[0,\infty)=\sigma(-\Delta)$ and a possible additional set of discrete eigenvalues  which can accumulate at $[0,\infty)$ only. A classical result in the study of these isolated eigenvalues for \emph{selfadjoint} Schr\"odinger operators are
 the Lieb-Thirring (L-T) inequalities, which state that for $V=\overline{V} \in L^p(\mr^d)$ with $p$ satisfying (\ref{eq:01}) one has
\begin{equation}
  \label{eq:lt}
  \sum_{\lambda \in \sigma_d(H), \lambda < 0} |\lambda|^{p-\frac d 2} \leq C(p,d) \|V_-\|_{L^p}^p,
\end{equation}
where $V_-=-\min(V,0)$ denotes the negative part of $V$. These inequalities were a major tool in Lieb and Thirring's proof of the stability of matter \cite{Lieb75} and the search for the optimal constants $C(p,d)$ remains an active field of current research. We refer to \cite{MR1775696, Hundertmark07} for more information on these topics.

In recent times, starting with work of Abramov, Aslanyan and Davies \cite{Abramov01}, there has also been an increasing interest in analogs of the L-T-inequalities for non-selfadjoint Schr\"odinger operators. For instance, Frank, Laptev, Lieb and Seiringer \cite{Frank06} considered the eigenvalues in sectors avoiding the positive half-line. By reduction to a selfadjoint problem (essentially doing what was sketched in Remark \ref{rem:Ouhabaz}) they showed that for $p \geq d/2  +1$ and $\chi > 0$
\begin{equation}\label{lt_frank}
  \sum_{\lambda \in \sigma_d(H), |\Im(\lambda)| \geq \chi \Re(\lambda)} |\lambda|^{p-\frac d 2} \leq C(p,d) \left( 1 + \frac 2 \chi \right)^p \|\Re(V)_-+ i\Im(V)\|_{L^p}^p.
\end{equation}
\begin{rem}
  By a suitable integration of inequality (\ref{lt_frank}) one can obtain an estimate on all discrete eigenvalues of $H$, see Corollary 3 in \cite{DHK08_2}. We will not discuss this result in this review.
\end{rem}
Concerning eigenvalues accumulating to $[0,\infty)$ Laptev and Safronov \cite{Laptev08} proved the following result:  If $\Re(V) \geq 0$ and $V \in L^p(\mr^d)$ for $p \geq 1$ if $d=1$ and $p > \frac d 2$ if $d \geq 2$ then
\begin{equation}\label{lt_laptev}
  \sum_{\lambda \in \sigma_d(H), \Re(\lambda) \geq 0 } \left( \frac{|\Im(\lambda)|}{|\lambda+1|^2+1} \right)^p \leq C(p,d) \|\Im(V)\|_{L^p}^p.
\end{equation}
Finally, let us also mention the recent work of Frank \cite{MR2820160}, which provides conditions for the boundedness of the eigenvalues of $H$ outside $[0,\infty)$, and the related works of Safronov \cite{MR2651940,MR2596049}.

Now let us have a look at what kind of L-T-inequalities we can obtain from Theorem \ref{cor:3} and \ref{thm:num_non2}, respectively,  and how these inequalities will compare to each other and to the inequalities (\ref{lt_frank}) and (\ref{lt_laptev}). We note that the results to follow can be regarded as refinements of our earlier work \cite{DHK08_2} (see also \cite{Hans_diss}).

We start with an application of Theorem \ref{cor:3}, where we require the stronger assumption that $M_V$ is $(-\Delta)$-compact (see Remark \ref{rem_relcomp} above).

\begin{thm}\label{thm:16}
Let $H=-\Delta+V$ be defined as above and let $\omega \leq 0$ be as defined in (\ref{eq:num_omega}). We assume that $V \in L^p(\mr^d)$ with $p \geq 2$ if $d\leq 3$ and $p > d/2$ if $d>4$. Then for $\tau \in (0,1)$ the following holds:
(i) If $\omega < 0$ and $p \geq d -\tau$ then
  \begin{equation}
  \sum_{\lambda \in \sigma_d(H)} \frac{\dist(\lambda,[0,\infty))^{p+\tau}}{|\lambda|^{\frac{d}{2}}(|\lambda|+|\omega|)^{2\tau}}
 \leq C(d,p,\tau) |\omega|^{-\tau} {\|V\|_{L^p}^p}.
  \end{equation}
(ii) If $\omega < 0$ and $p< d-\tau$ then
  \begin{equation}
  \sum_{\lambda \in \sigma_d(H)} \frac{\dist(\lambda,[0,\infty))^{p+\tau}}{|\lambda|^{\frac{p+\tau}{2}}(|\lambda|+|\omega|)^{\frac{d-p+3\tau}{2}}}
 \leq C(d,p,\tau) |\omega|^{-\tau} {\|V\|_{L^p}^p}.
  \end{equation}
(iii) If $\omega = 0$ then for $s> 0$
\begin{small}
  \begin{equation}
  \sum_{\lambda \in \sigma_d(H), |\lambda| > s} \frac{\dist(\lambda,[0,\infty))^{p+\tau}}{|\lambda|^{\frac d 2 + 2\tau}}
 + \sum_{\lambda \in \sigma_d(H), |\lambda| \leq s} \frac{\dist(\lambda,[0,\infty))^{p+\tau}}{|\lambda|^{\frac d 2}s^{2\tau}}
\leq C(d,p,\tau) s^{-\tau} \|V\|_{L^p}^p.
  \end{equation}
\end{small}
\end{thm}

Before presenting the proof of Theorem \ref{thm:16} let us consider what can be obtained by applying Theorem \ref{thm:num_non2} in the present context. Here, as compared to Theorem \ref{thm:16}, we don't need the relative compactness of $M_V$ but can (almost) stick to the more general assumption (\ref{eq:01}). However, now we require that $\Re(V) \geq \omega$ for some $\omega \leq 0$, which was not necessary in the previous result.
\begin{rem}
  Note that $\Re(V) \geq \omega$ is a sufficient but \emph{not} a necessary condition for $\num(H)$ being a subset of  $\{ \lambda: \Re(\lambda) \geq \omega \}$.
\end{rem}

\begin{thm}\label{thm:17}
Let $H=-\Delta+V$ be defined as above, where we assume that $V \in L^p(\mr^d)$ with $p \geq 1$ if $d=1$ and $p > d/2$ if $d\geq 2$. In addition, we assume that $\Re(V) \geq \omega$ for $\omega \leq 0$. Then for $\tau>0$ the following holds: (i) If $\omega<0$ then
\begin{equation}
\sum_{\lambda \in \sigma_d(H)} \frac{\dist(\lambda,[0,\infty))^p}{(|\lambda|+|\omega|)^{\frac d 2 +\tau}} \leq C(d,p,\tau) |\omega|^{-\tau} \|\Re(V)_- +i\Im(V)\|_{L^p}^p.
\end{equation}
(ii) If $\omega=0$ then for $s>0$
\begin{small}
\begin{equation}
\sum_{\lambda \in \sigma_d(H),|\lambda|>s} \frac{\dist(\lambda,[0,\infty))^p}{|\lambda|^{\frac d 2 +\tau}}
+ \sum_{\lambda \in \sigma_d(H), |\lambda|\leq s} \frac{\dist(\lambda,[0,\infty))^p}{|\lambda|^{\frac d 2 -\tau}s^{2\tau}} \leq C s^{-\tau} \|\Re(V)_- +i\Im(V)\|_{L^p}^p,
\end{equation}
\end{small}
where $C=C(d,p,\tau)$.
\end{thm}

As the reader might already have guessed, the comparison of the estimates obtained in the previous two theorems and the estimates  (\ref{lt_frank}) and (\ref{lt_laptev}) is a quite complex task, requiring
the analysis of a variety of different cases. However, we think it is better not to be too pedantic here, and so will restrict ourselves to a broad sketch of what is going on.

The first thing that is apparent is that the previous two theorems provide estimates which are not restricted to certain subsets of eigenvalues, as was the case with the estimates (\ref{lt_frank}) and (\ref{lt_laptev}). Concerning the $L^p$-assumptions on $V$, Theorem \ref{thm:17} and estimate (\ref{lt_laptev}) are less restrictive than the other two results; on the other hand, both Theorem \ref{thm:17} and estimate (\ref{lt_laptev}) require an additional assumption on the real part of the potential.  Concerning the right-hand sides of the inequalities, estimate (\ref{lt_frank}) stands out, since it is the only estimate which depends on $H$ only through the $L^p$-norm of the potential $V$, all other estimates also depending on $\omega=\omega(H)$. Whether this $\omega$-dependence is indeed necessary if one is considering all eigenvalues of $H$, not restricting oneself to eigenvalues outside sectors, is one among the many open questions on this topic.

Concerning the amount of information on the discrete eigenvalues that can be obtained from the different results, one has to distinguish between sequences of eigenvalues converging to some point in $(0,\infty)$ and to $0$, respectively, quite similarly to the case of Jacobi operators where we also had to distinguish between interior and boundary points of the essential spectrum. Suffice it to say that here, as in the Jacobi case, Theorem \ref{thm:16} (to be obtained via the complex analysis approach) is weaker than Theorem \ref{thm:17} (to be obtained via the operator-theory approach) concerning sequences of eigenvalues converging to some interior point of the essential spectrum $[0,\infty)$, whereas each of the results can be stronger than the other if one is considering eigenvalues converging to the boundary point $0$, depending on the
parameters involved.

\begin{problem}\label{p2}
All of the above results seem to suggest that the most natural generalization of the selfadjoint L-T-inequalities to the non-selfadjoint setting would be an estimate of the form
\begin{equation}\label{eq:goal}
  \sum_{\lambda \in \sigma_d(H)} \frac{\dist(\lambda,[0,\infty))^p}{|\lambda|^{\frac d 2}} \leq C(p,d)  \|V\|_{L^p}^p,
\end{equation}
with $p$ satisfying Assumption (\ref{eq:01}) (this is particularly true of the above estimates in case that $\omega=0$, just formally set $\tau=0$).  The validity or falsehood of estimate (\ref{eq:goal}), without any additional assumptions on $V$, can justly be regarded as one of  the major open problems in this field.\\
\end{problem}

\noindent It remains to present the proofs of Theorem \ref{thm:16} and Theorem \ref{thm:17}. Both  will rely on estimates on the $\mathcal{S}_p$-norm  of operators of the form $M_W(\lambda + \Delta)^{-1}$. Since $(\lambda + \Delta)^{-1}=F^{-1}M_{k_\lambda}F$, where
 \begin{equation}\label{eq:199}
 k_\lambda(x)=(\lambda-|x|^2)^{-1}, \quad x \in \mr^d,
 \end{equation}
as in the case of Jacobi operators this estimate will be reduced to an estimate on the $L^p$-norm of the bounded function $k_\lambda$. We will need the following three lemmas.

\begin{lemma}\label{lem:7}
  Let $d\geq1$. Then for $\lambda\in \mc\setminus [0,\infty)$ and $k_\lambda$ as defined in (\ref{eq:199}) the following holds: If $p > \max(d/2,1)$ then
\begin{equation} \label{eq:152}
\|k_\lambda\|_{L^p}^p\leq C(p,d)\frac{|\lambda|^{\frac d 2 - 1}}{\dist(\lambda,[0,\infty))^{p-1}}.
\end{equation}
\end{lemma}
\begin{proof}
For the elementary but quite lengthy proof we refer to \cite{Hans_diss}, page 103.
\end{proof}
The next result has already been hinted at in the study of Jacobi operators (see Lemma \ref{lem:simon}). See Simon \cite{b_Simon05}, Theorem 4.1, for a proof.
\begin{lemma} \label{lem:simon2}
Let $f, g \in L^p(\mr^d)$ where $p \geq 2$. Then the operator $M_f F^{-1}M_{g}F$ is in  $\mathcal{S}_p(L^2(\mr^d))$ and
$$ \|  M_f F^{-1}M_{g}F \|_{\mathcal{S}_p}^p \leq (2\pi)^{-d} \|f\|_{L^p}^p \|g\|_{L^p}^p.$$
\end{lemma}

Combining the previous two lemmas, we obtain a  bound on the $\mathcal{S}_p$-norm of $M_W(\lambda+\Delta)^{-1}$.
\begin{lemma}\label{lem:8} Let  $W \in L^p(\mr^d)$ where $p \geq 2$ if $d \leq 3$ and $p > d/2$ if $d \geq 4$. Then for $\lambda\in \mc\setminus [0,\infty)$ we have
\begin{equation}
\|M_W (\lambda+\Delta)^{-1}\|_{\mathcal{S}_p}^p\leq C(p,d) \|W\|_{L^p}^p  \frac{|\lambda|^{\frac d 2 - 1}}{\dist(\lambda,[0,\infty))^{p-1}}.
\end{equation}
\end{lemma}
We are now prepared for the
\begin{proof}[Proof of Theorem \ref{thm:16}]
We apply Theorem \ref{cor:3} with $H=-\Delta+M_V$ and $H_0=-\Delta$, taking estimate (\ref{eq:reso_es}) into account. With the notation of that theorem we obtain from the previous lemma
that $\alpha=p-1, \beta = \frac d 2 -1, C_0=1$ and $K=C(p,d) \|V\|_{L^p}^p$. All that remains is to compute the constants $\eta_0, \eta_1$ and $\eta_2$ appearing in Theorem \ref{cor:3}, treating the
cases $p \geq d-\tau$ and $p< d- \tau$ separately (and noting that by assumption $\tau \in (0,1)$).
\end{proof}

The proof of Theorem \ref{thm:17} is a little more involved.

\begin{proof}[Proof of Theorem \ref{thm:17}]
First of all we note that using an approximation argument it is sufficient to prove the theorem assuming that $V \in L^\infty_0(\mr^d)$, the bounded functions with compact support, see \cite[Lemma 5.4]{MR2836430} for more details. In particular, in this case $H=-\Delta+M_V$ since $M_V$ is $(-\Delta)$-compact.

So in the following let $V \in L_0^\infty(\mr^d)$ with $\Re(V) \geq \omega \; (\omega \leq 0)$ and let $H=-\Delta+M_V$ and $H_0 = -\Delta+ M_{\Re(V)_+}$. We are going to show that
for $a < \omega$ and $p\geq 1$ if $d=1$ or $p > \frac d 2$ if $d \geq 2$ we have
\begin{equation}
  \label{eq:fgh}
\| R_{H}(a)-R_{H_0}(a)\|_{\mathcal{S}_p}^p \leq C(p,d) \frac{1}{|a|^{p-\frac d 2}(|a|-|\omega|)^p} \|\Re(V)_-+i\Im(V)\|_p^p.
\end{equation}
If this is done an application of Theorem \ref{thm:num_non2} will conclude the proof.

As a first step in the proof of (\ref{eq:fgh}) we use the second resolvent identity to rewrite the resolvent difference as
\begin{eqnarray*}
 && R_H(a)-R_{H_0}(a) \nonumber \\
&=& (a-H)^{-1}(-a-\Delta)^{1/2}(-a-\Delta)^{-1/2} M_{|W|^{1/2}} M_{\operatorname{sign}(W)} \nonumber \\
&& M_{|W|^{1/2}}(-a-\Delta)^{-1/2}(-a-\Delta)^{1/2}(a-H_0)^{-1},
\end{eqnarray*}
where $W=-\Re(V)_-+i\Im(V)$ and $\operatorname{sign}(W)=W/|W|$. Note that $-\Delta-a \geq -a \geq 0$. We will show below that $(-a-\Delta)^{1/2}(a-H_0)^{-1}$ is bounded on $L^2(\mr^d)$ with
\begin{equation}
  \label{eq:25}
\|(-a-\Delta)^{1/2}(a-H_0)^{-1}\| \leq |a|^{-1/2}.
\end{equation}
Moreover, we will show that for the closure of $(a+H)^{-1}(a-\Delta)^{1/2}$, initially defined on $\operatorname{Dom}((-\Delta)^{1/2})=W^{1,2}(\mr^d)$, we have
\begin{equation}
  \label{eq:26}
  \|\overline{(a-H)^{-1}(-a-\Delta)^{1/2}}\| \leq \frac{|a|^{1/2}}{|a|-|\omega|}.
\end{equation}
Hence, using H\"older's inequality for Schatten norms, the unitarity of $M_{\sign(W)}$ and the fact that the Schatten norm of an operator and its adjoint coincide we obtain
\begin{eqnarray}
&& \|R_H(a)-R_{H_0}(a)\|_{\mathcal{S}_p}^p \nonumber \\
&\leq&  (|a|-|\omega|)^{-p} \|(-a-\Delta)^{-1/2} M_{|W|^{1/2}} M_{\operatorname{sign}(W)}M_{|W|^{1/2}}(-a-\Delta)^{-1/2} \|_{\mathcal{S}_p}^p \nonumber \\
&\leq& (|a|-|\omega|)^{-p} \| M_{|W|^{1/2}} (-a-\Delta)^{-1/2} \|_{\mathcal{S}_{2p}}^{2p}. \label{eq:Ff}
\end{eqnarray}
Since $p\geq1$ and $p > d/2$ we can then apply Lemma \ref{lem:simon2} and Lemma \ref{lem:7}  to obtain
\begin{eqnarray}
&&  \| M_{|W|^{1/2}} (-a-\Delta)^{-1/2} \|_{\mathcal{S}_{2p}}^{2p} = \|M_{|W|^{1/2}} F^{-1} |k_a|^{1/2} F \|_{\mathcal{S}_{2p}}^{2p}  \nonumber \\
&\leq& (2\pi)^{-d} \|W\|_{L^p}^p \|k_a\|_{L^{p}}^{p} \leq  C(p,d) \|W\|_{L^p}^p |a|^{d/2-p}. \label{eq:Fff}
\end{eqnarray}
\begin{rem}
  The validity of the last estimate for $p=1$ and $d=1$ (which is not contained in Lemma \ref{lem:7}) is easily established.
\end{rem}
The estimates (\ref{eq:Ff}) and (\ref{eq:Fff}) show the validity of (\ref{eq:fgh}). It remains to prove (\ref{eq:25}) and (\ref{eq:26}). To prove (\ref{eq:26}), let $f \in L^2(\mr^d)$ with $\|f\|=1$.  Then
  \begin{eqnarray*}
  &&  \|(-a-\Delta)^{1/2}(a-H^*)^{-1}f\|^2 \\
&=& -\langle f, (a-H^*)^{-1}f \rangle - \langle \overline{V} (a-H^*)^{-1}f, (a-H^*)^{-1}f \rangle.
  \end{eqnarray*}
Since $\Re(V) \geq \omega$ we obtain
  \begin{eqnarray}
  &&  \|(-a-\Delta)^{1/2}(a-H^*)^{-1}f\|^2 \nonumber \\
&=& -\Re(\langle f, (a-H^*)^{-1}f \rangle) - \Re(\langle \overline{V} (a-H^*)^{-1}f, (a-H^*)^{-1}f \rangle) \nonumber \\
&\leq& -\Re(\langle f, (a-H^*)^{-1}f \rangle) + |\omega| \| (a-H^*)^{-1}f\|^2 \nonumber \\ 
&\leq& \|(a-H^*)^{-1}\| + |\omega| \|(a-H^*)^{-1}\|^2 \nonumber \\
&\leq& \frac{1}{\dist(a,\overline{\num}(H^*))} + \frac{|\omega|}{\dist(a,\overline{\num}(H^*))^2} \nonumber \\
&\leq& \frac{1}{|a|-|\omega|}  +  \frac{|\omega|}{(|a|-|\omega|)^2}
= \frac{|a|}{(|a|-|\omega|)^2}. \label{eq:27}
  \end{eqnarray}
 But (\ref{eq:27}) implies (\ref{eq:26}) since $$\overline{(a-H)^{-1}(-a-\Delta)^{1/2}}=[(-a-\Delta)^{1/2}(a-H^*)^{-1}]^*.$$ The proof of (\ref{eq:25}) is similar (and even simpler) and is therefore omitted.
\end{proof}

\section{An outlook}\label{sec:outlook}

In this final section we would like to present a short list of possible extensions of the results discussed in this paper, and of some open problems connected to these results  which we think might be worthwhile to pursue.

\begin{enumerate}

\item The majority of results in this paper dealt with non-selfadjoint perturbations of selfadjoint operators, with a particular emphasis on the case where the spectrum of the unperturbed operator is an interval. This choice of operators was sufficient for the applications we had in mind, but there are also two more intrinsic reasons for this restriction. Namely, in this case the closure of the numerical range and the spectrum of the unperturbed operator coincide, which was necessary for a suitable application of the operator-theoretic approach. Moreover, given this restriction the (extended) resolvent set of the unperturbed operator is conformally equivalent to the unit disk, which was important for the complex analysis approach.

Recent developments suggest that the restriction to such operators is not really necessary and that both our methods can be applied in a much wider context. Concerning the operator-theory approach this is a consequence of the fact that estimate (\ref{eq:Kato3}) remains valid (for $p>1$) for arbitrary perturbations of selfadjoint operators (see \cite{Hansmann12}), without any restriction on the spectrum of the selfadjoint operator (i.e. it does not need to be an interval). Concerning the complex analysis approach it follows from the fact that our main tool, the result of Borichev, Golinskii and Kupin (Theorem \ref{thm:c05}) has been generalized to functions acting on finitely connected \cite{golinskii11} and more general domains \cite{golinskii12}. These new results will allow to analyze a variety of interesting operators (like, e.g., periodic Schr\"odinger operators perturbed by complex potentials), and they also lead to the question of the ultimate limits of applicability of our methods.

\item  We have seen that neither of the two methods for studying eigenvalues developed in this paper subsumes the other, in the sense that
each method allows us to prove some results which cannot, at least at the present stage of our knowledge, be obtained from the other. One may
thus wonder whether there is some `higher' viewpoint from which one could obtain all the results which are derived by the two methods. Since our
two methods seem to rely on different ideas, it is not at all clear what such a generalized approach would look like.

\item In Chapter \ref{sec:applications} we have applied our results to Jacobi and Schr\"odinger operators. Many opportunities exist
for applying the results to other concrete classes of operators, e.g. Jacobi-type operators in higher dimensions, systems of partial differential
equations, composition operators and so on. Each application might involve its own technical challenges, which might be interesting in themselves.

\item Many questions remain as to the optimality or sharpness of our results. Such questions are, of course, relative to the precise class of operators considered, and we refer particularly to Problem \ref{p1} regarding Jacobi operators and to Problem \ref{p2} regarding Schr\"odinger operators.
Moreover the question of optimality can be understood in two senses. In the narrow sense, for a particular inequality we want to know that
it cannot be strengthened with respect to the values of the exponents appearing in it. To obtain this it is sufficient to construct a single
 operator for which the distribution of eigenvalues is exactly as implied by the inequality, and no better. In a wider (and much more difficult) sense, one would like to know whether some inequalities completely characterize the possible set of eigenvalues of operators of a particular class
of operators. To show this, one must construct, for {\it{each}} set of complex numbers satisfying the inequality, an operator in the relevant class which has precisely this set of eigenvalues - that is solve an inverse problem. Techniques for constructing operators of certain classes with explicitly known spectrum would thus be very valuable.

\item Another direction which should be interesting and challenging is the generalization of results of the type considered here to operators on Banach spaces.
The notions of Schatten-class perturbations, of infinite determinants and of the numerical range,  which are all central for us, have generalizations to Banach spaces, so that one can hope that at least some of our results can be generalized. This might lead to further information on concrete classes of operators.

\item It should be mentioned that in spectral theory and its applications, the distribution of eigenvalues is only one aspect of interest, and one would also like to learn about the corresponding eigenvectors. In the case of non-selfadjoint operators, the eigenvectors are not orthogonal, and we do not have the spectral theorem which ensures that the Hilbert space is a direct sum of subspaces corresponding to the discrete and the essential spectrum. We would like to know more about the eigenvectors and the subspace generated by them.

\item A related direction somewhat removed from our work, but with which potential connections could be made, is the numerical computation of
eigenvalues of operators of the type that have been considered here. How should one go about in obtaining approximations of eigenvalues
of non-selfadjoint operators which are relatively compact perturbations of an operator with essential spectrum, and can some of the ideas used in our investigations (e.g. the perturbation determinant and complex analysis) be of use in the development of effective algorithms and/or in their analysis?
\end{enumerate}

\section*{List of important symbols} 

\begin{itemize}
\item[] $(.)_\pm$ - positive and negative part of a function/number
\item[] $\langle . , . \rangle$ - scalar product
\item[] $\bdd(\hil)$ - bounded linear operators on $\hil$
\item[] $\cld(\hil)$ - closed linear operators in $\hil$
\item[] $\hat{\mc}$ - extended complex plane
\item[] $\md$ - unit disk in the complex plane
\item[] $d_a,d_a^{Z,Z_0},d_\infty,d_\infty^{Z,Z_0}$ - perturbation determinant (Section \ref{sec:pert-determ})
\item[] $\dom(.)$ - domain of an operator/form
\item[] $(-\Delta)$ - Laplace operator in $L^2(\mr^d)$
\item[] $\hil$ - a complex separable Hilbert space
\item[] $H(\md)$ - holomorphic functions in the unit disk
\item[] $\ker(.)$ - kernel of a linear operator
\item[] $M_V,M_v$ - operator of multiplication by $V,v$ in $L^2(\mr^d), l^2(\mz)$
\item[] $\mathcal{M}, \mathcal{M}(\alpha,\vec{\beta}, \gamma, \vec{\xi}, K)$ - subclass of $H(\md)$ (see Definition \ref{def:c01})
\item[] $N(h,r)$ - number of zeros of $h \in H(\md)$ in closed disk of radius $r$
\item[] $\|.\|_{\mathcal{S}_p}$ - Schatten-$p$-norm
\item[] $\num(.)$ - numerical range of a linear operator
\item[] $P_Z, P_Z(\lambda)$ - Riesz projection
\item[] $\ran(.)$ - range of a linear operator
\item[] $\rank(.)$ - rank of a linear operator
\item[] $R_Z(\lambda)=(\lambda-Z)^{-1}$ - the resolvent
\item[] $\mr_+$ - the interval $[0,\infty)$
\item[] $\rho(.), \hat{\rho}(.)$ - (extended) resolvent set of a linear operator
\item[] $\mathcal{S}_\infty(\hil)$ - compact linear operators on $\hil$
\item[] $\mathcal{S}_p(\hil)$ - Schatten class of order $p$
\item[] $\sigma(.), \sigma_d(.), \sigma_{ess}(.)$- spectrum (discrete, essential) of a linear operator
\item[] $\mt$ - unit circle in the complex plane
\item[] $(\mt^N)_*$ - subset of $\mt^N$ (see Definition \ref{eq:22})
\item[] $\dotcup$ - a disjoint union
\item[] $\mathcal{Z}(.)$ - zero set of a function
\end{itemize}

\bibliographystyle{plain}
\bibliography{Bibliography}

\end{document}